\documentclass{amsart}
\usepackage{amssymb}
\usepackage{mathrsfs}
\usepackage{graphicx}
\usepackage[foot]{amsaddr}
\usepackage{url}
\usepackage[table]{xcolor}
\usepackage{ytableau}
\usepackage{tikz}

\usepackage[letterpaper,top=2cm,bottom=2cm,left=3cm,right=3cm,marginparwidth=1.75cm]{geometry}

\usepackage[colorlinks=true, allcolors=blue]{hyperref}

\definecolor{ltblue}{RGB}{168,195,251}
\definecolor{ltpink}{RGB}{240,178,169}
\definecolor{ltgreen}{RGB}{191,226,180}

\newtheorem{theorem}{Theorem}[section]

\newtheorem{corollary}[theorem]{Corollary}
\newtheorem{lemma}[theorem]{Lemma}
\theoremstyle{remark}
\newtheorem{definition}[theorem]{Definition}
\newtheorem{remark}[theorem]{Remark}
\newtheorem{example}[theorem]{Example}
\newtheorem{problem}[theorem]{Problem}

\numberwithin{equation}{section}

\begin{document}

\title{Totally symmetric Grassmannian codes}
\author{Matthew~Fickus$^*$ 
\quad Joseph~W.~Iverson$^\dag$ 
\quad John~Jasper$^*$ 
\quad Dustin~G.~Mixon$^\ddag$$^\S$}
\address{$^*$Department of Mathematics and Statistics, Air Force Institute of Technology, Wright-Patterson AFB, OH}
\address{$^\dag$Department of Mathematics, Iowa State University, Ames, IA}
\address{$^\ddag$Department of Mathematics, The Ohio State University, Columbus, OH}
\address{$^\S$Translational Data Analytics Institute, The Ohio State University, Columbus, OH}

\email{jwi@iastate.edu}

\begin{abstract}
We introduce a general technique to construct tight fusion frames with prescribed symmetries.
Applying this technique with a prescription for ``all the symmetries'', we construct a new family of equi-isoclinic tight fusion frames (EITFFs), which consequently form optimal Grassmannian codes.
By virtue of their construction, our EITFFs have the remarkable property of \textit{total symmetry}: any permutation of subspaces can be achieved by an appropriate unitary.
\end{abstract}

\maketitle

\section{Introduction}

This paper\footnote{An early version of this paper appeared on the arXiv with the title ``Equi-isoclinic subspaces from symmetry''. To help distinguish between the two (very different) versions, the authors also changed the title.} concerns a packing problem in Grassmannian space: in a $d$-dimensional ambient space, how can one arrange $n$ subspaces of dimension $r\geq 1$ to be well separated, in the sense that 
any pair of lines through the origin in distinct subspaces casts a wide angle?
To be more precise, assume $d \geq r \geq 1$ and $n \geq 2$, let $V$ be a $d$-dimensional Hilbert space over $\mathbb{F} \in \{ \mathbb{R}, \mathbb{C} \}$, and let $\{W_j\}_{j\in [n]}$ be a sequence of $r$-dimensional subspaces $W_j \leq V$, where $[n]:=\{1,\ldots,n\}$.
For $i,j \in [n]$, consider the smallest angle between a line in $W_i$ and a line in $W_j$, namely, $\theta_{i,j} \in [0,\tfrac{\pi}{2}]$ with
\[
\cos \theta_{i,j}= \operatorname{max} \big\{ | \langle w_i, w_j \rangle | : w_i \in W_i, \, w_j \in W_j, \, \| w_i \| = \| w_j \| = 1 \big\}.
\]
Then the \textbf{spectral distance} is defined as
\[
\operatorname{dist}_s(W_i,W_j) := \sin \theta_{i,j}.
\]
Accounting for all pairs of subspaces in the sequence, the \textbf{spectral coherence} (or \textit{block coherence})
\[
\mu_s(\{W_j\}_{j\in [n]}) := \sqrt{1 - \min_{i \neq j} \operatorname{dist}_s(W_i,W_j)^2 }
= \max_{i \neq j} \cos \theta_{i,j}
\]
is the cosine of the smallest angle between any two lines in distinct subspaces.
The smaller the spectral coherence, the better the subspace separation.

In this paper, we construct subspace ensembles that minimize spectral coherence by attaining the (spectral) \textbf{Welch bound}~\cite{Welch:74,DHST:08,CTX:15}:
\begin{equation}
\label{eq: spectral welch}
\mu_s(\{W_j\}_{j\in [n]}) \geq \sqrt{ \tfrac{nr - d}{d(n-1)} },
\end{equation}
which holds whenever $nr \geq d$ (a safe assumption since we can replace $V$ with $W_1+\dotsb+W_n$).
To describe conditions for equality, let $P_j$ denote orthogonal projection onto $W_j$ for $j \in [n]$.
Then equality holds in~\eqref{eq: spectral welch} if and only if
	\begin{itemize}
	\item[(i)]
	$\exists \, \alpha$ s.t.\ for all $i \neq j$ in $[n]$, every $w \in W_j$ satisfies $\| P_i w \| = \alpha \| w \|$, and
	\item[(ii)]
	$\exists \, c$ s.t.\ $P_1+\dotsb+P_n = c I$.
	\end{itemize}
When (i) holds, $\{W_j\}_{j \in [n]}$ is called \textbf{equi-isoclinic} with parameter~$\alpha$; when (ii) holds, $\{W_j\}_{j\in [n]}$ is called a \textbf{tight fusion frame}.\footnote{Warning: some texts instead define $\alpha^2$ to be the isoclinism parameter.}
When both conditions hold, so that the Welch bound is attained, $\{W_j\}_{j \in[n]}$ is called an \textbf{equi-isoclinic tight fusion frame}, abbreviated $\operatorname{EITFF}_{\mathbb{F}}(d,r,n)$; then the isoclinism parameter is necessarily
\begin{equation}
\label{eq: isoclinism parameter}
\alpha = \sqrt{ \tfrac{nr - d}{d(n-1)} },
\end{equation}
and the fusion frame constant is necessarily $c = \tfrac{nr}{d}$.

To reiterate, EITFFs are optimal Grassmannian codes with respect to spectral distance.
However, they are difficult to construct.
The case $r=1$ is understood best; in this regime, an $\operatorname{EITFF}_{\mathbb{F}}(d,1,n)$ is also known as an \textbf{equiangular tight frame}, abbreviated $\operatorname{ETF}_{\mathbb{F}}(d,n)$.
ETFs form optimal projective codes, and many constructions are known~\cite{FM:15}.
Next, a simple tensor construction converts any $\operatorname{EITFF}_{\mathbb{F}}(d_0,r_0,n)$ into an $\operatorname{EITFF}_{\mathbb{F}}(d_0m,r_0m,n)$ for any choice of $m \geq 1$~\cite{LS:73a,CTX:15,FMW:21}.
In particular, an $\operatorname{ETF}_{\mathbb{F}}(d_0,n)$ can be used to create an $\operatorname{EITFF}_{\mathbb{F}}(d_0m,m,n)$ for any $m \geq 1$.
While ETFs are themselves rare, this technique produces an $\operatorname{EITFF}_{\mathbb{F}}(d,r,n)$ for a variety of parameters, all with the property that $r$ divides $d$ and an $\operatorname{ETF}_{\mathbb{F}}(\tfrac{d}{r},n)$ exists.
It is thus of interest to find an $\operatorname{EITFF}_{\mathbb{F}}(d,r,n)$ with $\tfrac{d}{r} \notin \mathbb{Z}$.
Some examples are known~\cite{GH:92,ET:06,ETR:09,IKM:21,FIJM:22,FIJM:23}, but to the authors' knowledge, all but finitely many of them have $\tfrac{d}{r} \in \tfrac{1}{2} + \mathbb{Z}$ since they ultimately derive from an $\operatorname{EITFF}_{\mathbb{F}}(d_0,2,n)$ with $d_0$ odd.
The main result of this paper, Theorem~\ref{thm: single layer}, provides infinitely many new examples by constructing an $\operatorname{EITFF}_{\mathbb{F}}(d,r,n)$ for infinitely many choices of the coset $\tfrac{d}{r} + \mathbb{Z}$.

While we are primarily motivated by the Grassmannian packing problem above, our results have a secondary impact for algebraic combinatorics, which has a historical interest in equi-isoclinic subspaces~\cite{LS:73a}.
The main problem in this area is to determine the maximum number of equi-isoclinic $r$-dimensional subspaces with isoclinism parameter $\alpha$ in an ambient space of dimension $d$ over $\mathbb{F}$.
Denoting $N_\mathbb{F}(d,r,\alpha)$ for this quantity, the \textbf{relative bound}~\cite{LS:73a} states that
\[
N_\mathbb{F}(d,r,\alpha) \leq \tfrac{d(1-\alpha^2)}{r-d\alpha^2}=:n,
\]
with equality if and only if an $\operatorname{EITFF}_{\mathbb{F}}(d,r,n)$ exists.
Since Theorem~\ref{thm: single layer} produces infinitely many new EITFF parameters, it resolves infinitely many cases of this problem.

Our technique pulls a classic maneuver from both the packing and algebraic combinatorics playbooks: take advantage of symmetry.
To elaborate, the \textbf{automorphism group} of a subspace ensemble $\{W_j\}_{j \in [n]}$ consists of all permutations $\sigma \in S_n$ for which there exists a unitary $U$ such that $UW_j = W_{\sigma(j)}$ for every $j\in [n]$.
Each EITFF constructed by Theorem~\ref{thm: single layer} is \textbf{totally symmetric}, meaning its automorphism group is all of $S_n$.\footnote{A preliminary version of this result appeared in the conference proceeding~\cite{FIJM:ICASSP}.}
More broadly, Theorem~\ref{thm: one layer general group} provides a very general way to build a tight fusion frame with prescribed automorphisms, and the EITFFs of Theorem~\ref{thm: single layer} are instances of this construction, which ultimately relies on representation theory.
In Subsection~\ref{subsec: exs}, we use Theorem~\ref{thm: one layer general group} to build some other sporadic examples of new EITFFs with remarkable symmetry.

Given the definition of equi-isoclinism, it might not be surprising to learn that symmetry plays a role in the construction of some EITFFs.
In fact, our results might be seen as part of a bigger story about symmetry interacting with optimality.
Fejes T\'{o}th observed this interplay in his work on regular figures~\cite{FT:64}, and Conway, Hardin, and Sloane noted it again in their seminal paper on Grassmannian codes~\cite{CHS:96}.
For projective codes, symmetry implies optimality since any ensemble of $n>d$ lines spanning $\mathbb{F}^d$ with a doubly transitive automorphism group forms an ETF~\cite{C:18,IM:DTI,IM:DTII,DK:23}.
In the other direction, optimality implies symmetry for any $\operatorname{EITFF}_{\mathbb{F}}(d,r,n)$ with $\tfrac{d}{r} = 2$, since its automorphism group necessarily contains $A_n$~\cite{FGLI:24}.
(None of the constructions in this paper have $\tfrac{d}{r} = 2$.)
By comparison with these implications, one might say that symmetry \textit{facilitates} optimality in this paper, since a transitive permutation group is the first ingredient in the construction of Theorem~\ref{thm: single layer}.
This follows an established theme in the literature.
Modern frame theory was founded on group-theoretic techniques~\cite{DGM:86}; more recently, results akin to Theorem~\ref{thm: single layer} appear for tight fusion frames with abelian symmetry in~\cite{FIJM:23} and for tight frames with prescribed symmetry in~\cite{IJM:20}.
Furthermore, symmetry is famously conjectured to facilitate optimality in \textit{Zauner's conjecture} on the existence of an $\operatorname{ETF}_{\mathbb{C}}(d,d^2)$ for every~$d\geq1$~\cite{Z:99,RBSC:04}.
(Interestingly, symmetry may also \textit{obstruct} optimality for projective codes, since any $\operatorname{ETF}_{\mathbb{F}}(d,n)$ with a triply transitive automorphism group is necessarily a trivial construction with $d \in \{1,n-1,n\}$~\cite{K:19,IM:DTII}.
As our results demonstrate, this obstruction evaporates for EITFFs with higher-dimensional subspaces.)

Here, we use symmetry as a means of constructing optimal Grassmannian codes, but its presence in our constructions has the potential for further applications in \textit{block compressed sensing}~\cite{T:04,EKB:10,CTX:15}.
In detail, let $\{ W_j \}_{j \in [n]}$ be as above, and for each $S \subseteq [n]$, define $\delta_S$ as the largest distance between~1 and any of the top $r|S|$ eigenvalues of $\sum_{j \in S} P_j$.
Then for $k \leq \tfrac{d}{r}$, the $k$th \textit{block restricted isometry constant} may be expressed as
\[
\delta_k := \max_{\substack{S \subseteq [n], \\ |S| = k}} \delta_S.
\]
(This is not the usual definition, but it is equivalent.)
In settings where $\delta_k$ is sufficiently small, the Grassmannian code allows one to recover arbitrary block-sparse signals of block sparsity $k$ from $d$ linear measurements that are determined by the code. 
Unfortunately, it is computationally prohibitive to compute $\delta_k$ in general, and so researchers have turned to random matrix theory to close the gap. 
However, if the code is totally symmetric, then $\delta_S$ is the same for every $S \subseteq [n]$ of size $k$, and so $\delta_k$ can be computed with ease. 
While this is part of our motivation, we leave its investigation for follow-up work.

The paper is organized as follows.
Following preliminaries in Section~\ref{sec: preliminaries}, we provide our general construction of tight fusion frames with prescribed automorphisms in Section~\ref{sec: tight fusion frames from symmetry}.
This result has two versions: the ``single layer'' version of Theorem~\ref{thm: one layer general group} is simpler and uses irreducible representations, while the ``multiple layer'' version of Theorem~\ref{thm: multiple layers general group} is more general and allows for reducible (but still multiplicity-free) representations.
The remainder of the paper interprets the constructions of Section~\ref{sec: tight fusion frames from symmetry} for EITFFs prescribed to have total symmetry.
Section~\ref{sec: Sn review} reviews the necessary representation theory of symmetric groups, 
and then Section~\ref{sec: total symmetry one layer} uses this theory to build totally symmetric EITFFs from the ``single layer'' construction (Theorem~\ref{thm: single layer}).
Finally, Section~\ref{sec: total symmetry multiple layers} develops the theory of totally symmetric EITFFs from multiple-layer constructions, leading to infinitely many additional examples (Theorems~\ref{thm: three parts} and~\ref{thm: four parts}).

\section{Preliminaries}
\label{sec: preliminaries}

We begin by establishing notation and terminology.

\subsection{Group actions}

The reader may consult~\cite{DM:96} for background on group actions.
We use the following conventions.
Given a set $X$, we write $S_X$ for the group of all bijections $g \colon X \to X$, where the group operation is given by composition, with functions on the \textit{right} acting first: $fg = f \circ g$.
The \textbf{symmetric group} on $n$ points is $S_n := S_{[n]}$.
We adopt the usual cycle notation for its elements, where as above permutations on the right act first.
Thus, $(1\, 2)(2\, 3) = (1\, 2\, 3)$.
(Warning: GAP~\cite{gap} implements the opposite group operation!)
Whenever it is convenient we abuse notation and consider $S_{n-1}$ to be a subgroup of $S_n$, where elements of $S_{n-1}$ fix the point $n$ and permute the elements of $[n-1] \subset [n]$ as expected.

For a group $G$ and a set $X$, a (left) \textbf{group action} of $G$ on $X$ amounts to a group homomorphism $\sigma \colon G \to S_X$.
Given such an action, we typically suppress $\sigma$ in our notation and abbreviate $g x = \sigma(g)x$ for $g \in G$ and $x \in X$.
The action is called \textbf{faithful} if $\sigma$ is injective, i.e., no nontrivial element of $G$ fixes every element of $X$.
For non-faithful actions the image $\sigma(G) \leq S_X$ is isomorphic to $G/\ker \sigma$.
The \textbf{orbit} of a point $x \in X$ is the subset $G x := \{ g  x : g \in G\} \subseteq X$; the orbits partition $X$.
The action is \textbf{transitive} if $X = G x$ for some (hence any) $x \in X$, or equivalently, for every $x,y \in X$ there exists $g \in G$ with $g  x = y$.
In particular, if $G$ acts transitively on $X$ then for any choice of $x_0 \in X$ there exists a \textbf{transversal} $\{t_x\}_{x\in X}$ of elements in $G$ such that $t_x x_0 = x$ for every $x \in X$.
The \textbf{stabilizer} of a point $x \in X$ is the subgroup $G_x := \{ g \in G : g  x = x \} \leq G$.
Given any subgroup $H \leq G$, $G$ acts transitively on the left cosets $G/H$ via $g  (hG) = ghG$, where the stabilizer of the coset $H \in G/H$ is the subgroup $H\leq G$.
Up to equivalence, this accounts for all transitive group actions: if $G$ acts transitively on $X$ then for any choice of $x_0 \in X$ the mapping $f \colon G/G_{x_0} \to X$ given by $f(gH) = gx_0$ is a well-defined bijection that satisfies $f(hgG_{x_0})=hf(g G_{x_0})$ for every $h \in G$ and $gG_{x_0} \in G/G_{x_0}$.
The action of $G$ on $X$ is \textbf{regular} if it is transitive and some (hence every) stabilizer $G_{x_0}$ is trivial; equivalently, the mapping $g \mapsto g  x_0$ gives a bijection $G \to X$.
We say $G$ acts \textbf{doubly homogeneously} on $X$ if for every two unordered pairs $\{x,y\},\{x',y'\} \subseteq X$ with $x \neq y$ and $x' \neq y'$ there exists $g \in G$ such that $\{g  x, g  y\} = \{ x', y' \}$.
The action is \textbf{doubly transitive} if for every two ordered pairs $(x,y)$ and $(x',y')$ of elements in $X$ with $x \neq y$ and $x' \neq y'$ there exists $g \in G$ such that $(g  x, g  y) = ( x', y' )$.

\subsection{Subspace ensembles}

Assume $d \geq 1$, and let $V$ be a $d$-dimensional Hilbert space over $\mathbb{F} \in \{ \mathbb{R}, \mathbb{C} \}$, where the inner product $\langle \cdot , \cdot \rangle$ is linear in the first variable and either linear or conjugate-linear in the second variable, according as $\mathbb{F} = \mathbb{R}$ or $\mathbb{C}$.
We use the following terminology for an ensemble $\mathcal{W} = \{ W_j \}_{j\in [n]}$ of $r$-dimensional subspaces $W_j \leq V$, where $n,r \geq 1$.
For each $j \in [n]$, let $P_j \colon V \to V$ give orthogonal projection onto $W_j$.
Then $S = \sum_{j=1}^n P_j$ is the \textbf{fusion frame operator} of $\mathcal{W}$, which equals a scalar multiple of the identity if and only if $\mathcal{W}$ forms a tight fusion frame for $V$, abbreviated $\operatorname{TFF}_{\mathbb{F}}(d,r,n)$.
(No confusion should arise between the fusion frame operator $S$ and the symmetric group $S_n$.)
Next, fix an $r$-dimensional Hilbert space $W$ over $\mathbb{F}$, and choose linear isometries $\Phi_j \colon W \to V$ with $\operatorname{im} \Phi_j = W_j$ for each $j \in [n]$, so that each $P_j =  \Phi_j \Phi_j^*$.
Given $i,j \in [n]$, $\Phi_i^* \Phi_j \colon W \to W$ is called a \textbf{cross-Gram operator} for $W_i$ and $W_j$.
By fixing an orthonormal basis for $W$, we may identify $\Phi_i^* \Phi_j$ with an $r \times r$ matrix, called a \textbf{cross-Gram matrix} for $W_i$ and $W_j$.
Now fix an orthonormal basis for $V$.
Then each $\Phi_j$ may be identified with a $d \times r$ matrix whose columns form an orthonormal basis for $W_j$, and $\mathcal{W}$ can be described by a \textbf{fusion synthesis matrix} $\Phi := \left[ \begin{array}{ccc} \Phi_1 & \cdots & \Phi_n \end{array} \right] \in \mathbb{F}^{d \times nr}$.
Here, $\Phi \Phi^* \in \mathbb{F}^{d\times d}$ is the matrix of the fusion frame operator $S$, while $\Phi^*\Phi \in \mathbb{F}^{nr \times nr}$ is known as a \textbf{fusion Gram matrix} for $\mathcal{W}$ (note it is not unique).
Specifically, $\Phi^* \Phi$ is the $n \times n$ block array whose $(i,j)$ block is the cross-Gram matrix $\Phi_i^* \Phi_j \in \mathbb{F}^{r\times r}$.
Since $\Phi \Phi^*$ and $\Phi^* \Phi$ have the same nonzero eigenvalues with multiplicities, $\mathcal{W}$ is a tight fusion frame for $V$ if and only if $\Phi^* \Phi$ has exactly one nonzero eigenvalue $\lambda$ occurring with multiplicity~$d$.
Taking a trace shows the only option is $\lambda = \tfrac{nr}{d}$.
Furthermore, if $\mathcal{W}$ is a tight fusion frame and $d < nr$ then $\tfrac{nr}{nr-d}(I-\tfrac{d}{nr} \Phi^* \Phi)$ is the fusion Gram matrix of some $\operatorname{TFF}_{\mathbb{F}}(nr-d,r,n)$, called a \textbf{Naimark complement} for $\mathcal{W}$.
In the case where $\mathcal{W}$ is an $\operatorname{EITFF}_{\mathbb{F}}(d,r,n)$, any Naimark complement of $\mathcal{W}$ is an $\operatorname{EITFF}_{\mathbb{F}}(nr-d,r,n)$.

Given another $d$-dimensional Hilbert space $V'$ over $\mathbb{F}$ and another ensemble $\mathcal{W}' = \{ W_j' \}_{j\in [n]}$ of $r$-dimensional subspaces $W_j' \leq V'$, we say $\mathcal{W}$ and $\mathcal{W}'$ are \textbf{unitarily equivalent} if there exists a unitary $U \colon V \to V'$ with $U W_j = W_j'$ for every $j \in [n]$.
Equivalently, $\mathcal{W}$ and $\mathcal{W}'$ are unitarily equivalent if and only if some fusion Gram matrix for $\mathcal{W}$ is also a fusion Gram matrix for $\mathcal{W}'$.
In particular, if $\mathcal{W}$ is a tight fusion frame and $d < nr$ then any two Naimark complements for $\mathcal{W}$ are unitarily equivalent.
For this reason, we often refer to ``the'' Naimark complement of a tight fusion frame.

Many notions of distance in Grassmannian space can be expressed in terms of the \textbf{principal angles} between subspaces $W_i$ and $W_j$, defined as follows.
The first principal angle is the sharpest angle achieved between some unit vectors $w_i^{(1)} \in W_i$ and $w_j^{(1)} \in W_j$; explicitly,
\[
\theta_{i,j,1} := \min \big\{ \arccos | \langle w_i^{(1)}, w_j^{(1)} \rangle | : w_\ell^{(1)} \in W_\ell, \, \| w_\ell^{(1)} \| =  1, \, \ell \in \{i,j\} \big\}.
\]
The second principal angle $\theta_{i,j,2}$ is then the sharpest angle achieved between unit vectors $w_i^{(2)} \in W_i \ominus \operatorname{span}\{w_i^{(1)}\}$ and $w_j^{(2)} \in W_j \ominus \operatorname{span}\{w_j^{(1)}\}$ in the orthogonal complements, and so on, for a total of $r$ principal angles $\theta_{i,j,1},\dotsc,\theta_{i,j,r}$ between $W_i$ and $W_j$.
These can be expressed in terms of singular values of any cross-Gram operator $\Phi_i^* \Phi_j$, where each $\theta_{i,j,k} = \arccos \sigma_k(\Phi_i^* \Phi_j)$.
The subspaces $W_i$ and $W_j$ are called \textbf{isoclinic} with parameter $\alpha$ if $\theta_{i,j,k} = \arccos \alpha$ for every $k \in [r]$, that is, $(\Phi_i^* \Phi_j)^*(\Phi_i^* \Phi_j) = \alpha^2 I$; one can show this is equivalent to the condition that $\| \Phi_i \Phi_i^* w \| = \alpha \|w \|$ for every $w \in W_j$.
(Thus, ``equi-isoclinic with parameter $\alpha$'' means ``pairwise isoclinic with parameter $\alpha$''.)
In addition to spectral distance, we will also consider the \textbf{chordal distance} between subspaces, defined as
\[
\operatorname{dist}_c(W_i,W_j) := \sqrt{ \sum_{k \in [r]} \sin^2 \theta_{i,j,k} } = \sqrt{ r - \| \Phi_i^* \Phi_j \|_F^2 },
\]
where $\| \cdot \|_F$ indicates Frobenius norm.
The spectral distance can be expressed similarly as
\[
\operatorname{dist}_s(W_i,W_j) = \min_{k \in [r]} \sin \theta_{i,j,k} = \sqrt{ 1 - \| \Phi_i^* \Phi_j \|^2_{\text{op}} },
\]
where $\| \cdot \|_{\text{op}}$ indicates operator norm.
The minimum chordal distance obeys the (chordal) \textbf{Welch bound}~\cite{Welch:74,CHS:96,DHST:08}, which states that
\begin{equation}
\label{eq: chordal welch}
\min_{i \neq j} \operatorname{dist}_c(W_i,W_j) \leq \sqrt{ \frac{ nr(d-r) }{d(n-1)} }.
\end{equation}
Equality holds in~\eqref{eq: chordal welch} if and only if $\{ W_j \}_{j \in [n]}$ is an \textbf{equi-chordal tight fusion frame} for $V$, abbreviated $\operatorname{ECTFF}_{\mathbb{F}}(d,r,n)$, meaning $\{ W_j \}_{j \in [n]}$ is a tight fusion frame and $\operatorname{dist}_c(W_i,W_j)$ is equal across all pairs $i \neq j$.
When $r = 1$, the notions of ECTFF and EITFF coincide, and each is equivalent to the notion of an equiangular tight frame.
When $r > 1$, every EITFF is an ECTFF, but not vice versa.
As with EITFFs, any Naimark complement of an $\operatorname{ECTFF}_{\mathbb{F}}(d,r,n)$ is an $\operatorname{ECTFF}_{\mathbb{F}}(nr-d,r,n)$.

In addition to the notation above, we will find it convenient to work with subspace ensembles indexed by an arbitrary finite set $X$, not necessarily $[n]$.
Then similar notation and terminology apply.
In particular, the automorphism group of an ensemble $\{ W_x \}_{x\in X}$ of subspaces $W_x \leq V$ is the subgroup $\operatorname{Aut} (\{ W_x \}_{x \in X}) \leq S_X$ consisting of all permutations $\sigma$ for which there exists a unitary $U \in \operatorname{U}(V)$ satisfying $U W_x = W_{\sigma(x)}$ for every $x \in X$.

\subsection{Representation theory}

The reader may consult~\cite{CSST:18,CSST:10,S:77} for background on representation theory.
Fix a finite group $G$ and a field $\mathbb{F} \in \{ \mathbb{R}, \mathbb{C} \}$.
A (unitary) \textbf{representation} of $G$ is a group homomorphism $\rho \colon G \to \operatorname{U}(V)$, where $V$ is a finite-dimensional Hilbert space over $\mathbb{F}$.
Its \textbf{degree} is the dimension of $V$.
A subspace $W \leq V$ is \textbf{invariant} for $\rho(G)$ if $\rho(g) W = W$ for every $g \in G$; equivalently, orthogonal projection onto $W$ commutes with $\rho(g)$ for every $g \in G$.
Any invariant subspace $W \leq V$ gives rise to another representation $G \to \operatorname{U}(W)$ by restricting the domain and codomain of each $\rho(g)$ to $W$.
We say that $\rho$ is \textbf{reducible} if there exists a proper nontrivial invariant subspace $W \leq V$; otherwise it is \textbf{irreducible}.

Next, let $W$ be another Hilbert space over $\mathbb{F}$, and let $\pi \colon G \to \operatorname{U}(W)$ be another representation of $G$.
A linear map $L \colon W \to V$ \textbf{intertwines} $\pi$ and $\rho$ if $L \pi(g) = \rho(g)L $ for every $g \in G$.
If there exists an intertwining isometry $W \to V$, then we call $\pi$ a \textbf{constituent} of $\rho$.
We say $\pi$ and $\rho$ are \textbf{equivalent}, written $\pi \cong \rho$, if there exists an intertwining unitary $W \to V$; otherwise they are \textbf{inequivalent}.
In case $\pi$ is irreducible, its \textbf{isotypic component} in $\rho$ is the sum of all images of intertwining isometries $W \to V$.
(If no such isometries exist, i.e.\ $\pi$ is not a constituent of $\rho$, then the isotypic component is defined to be $\{0\}$.)
This is an invariant subspace for $\rho(G)$.
Part of \textbf{Schur's Lemma} says that if $\pi$ and $\rho$ are inequivalent and irreducible, then any linear map intertwining them is constantly~$0$.

Let $X$ be a finite set, let $\{ V_x \}_{x \in X}$ be a collection of finite-dimensional Hilbert spaces over $\mathbb{F}$, and let $\{ \pi_x \}_{x\in X}$ be a collection of representations $\pi_x \colon G \to \operatorname{U}(V_x)$.
Their \textbf{direct sum} $\bigoplus_{x\in X} \pi_x$ is the representation of $G$ on the orthogonal direct sum $\bigoplus_{x \in X} V_x$ given by $g \mapsto \bigoplus_{x\in X} \pi_x(g)$ for $g \in G$.
It is always possible to decompose $\rho$ as a direct sum of irreducible representations.
(Here and throughout, we freely conflate internal and external direct sums of Hilbert spaces.)
If the irreducible representations in some (hence any) such decomposition are pairwise inequivalent, then $\rho$ is called \textbf{multiplicity free}.

Finally, in the case where $\mathbb{F} = \mathbb{R}$, $\rho$ extends to a representation $\rho'$ on the complexification $V' \supset V$ of $V$, where any orthonormal basis for $V$ (over $\mathbb{R}$) is also an orthonormal basis for $V'$ (over $\mathbb{C}$), and the action of each $\rho'(g)$ on this orthonormal basis is identical.
Put differently, the matrix of each $\rho'(g)$ is exactly the matrix of $\rho(g)$, now understood to have complex entries.
We say $\rho$ is \textbf{absolutely irreducible} if $\rho'$ is irreducible.
Any irreducible representation over $\mathbb{C}$ is also called absolutely irreducible.

\section{Tight fusion frames from symmetry}
\label{sec: tight fusion frames from symmetry}

In this section, we develop our main technical tools for constructing tight fusion frames with prescribed transitive symmetry.
We have two such tools: Theorem~\ref{thm: one layer general group} implements a simpler ``single layer'' construction based on irreducible representations, while Theorem~\ref{thm: multiple layers general group} implements a more flexible ``multiple layer'' construction based on multiplicity-free representations.
We begin with the simpler of the two.

\subsection{One layer}

As inspiration for the following, let $G \leq S_X$ be contained in the automorphism group of an ensemble $\{ W_x \}_{x\in X}$ of subspaces of $V$.
Then for each $\sigma \in G$, there exists a unitary $U \in \operatorname{U}(V)$ such that $U W_x = W_{\sigma(x)}$ for every $x \in X$.
In particular, if $\sigma$ belongs to the stabilizer of a point $x_0 \in X$ then $U$ must hold $W_{x_0}$ invariant.
This suggests a connection with representation theory, which we leverage as a source of unitaries below.

\begin{theorem}
\label{thm: one layer general group}
Let $G$ be a finite group acting transitively on a set $X$ via $\sigma \colon G \to S_X$.
Fix a base point $x_0 \in X$ with stabilizer $H \leq G$, and select a transversal $\{ t_x \}_{x\in X}$ of elements in $G$ such that $\sigma(t_x) x_0 = x$ for each $x \in X$.
Next, let $\rho \colon G \to \operatorname{U}(\mathbb{F}^d)$ be an irreducible representation, where $\mathbb{F} \in \{ \mathbb{R},\mathbb{C}\}$, and suppose $W \leq \mathbb{F}^d$ is a subspace held invariant by $\rho(H)$.
Then the orbit $\mathcal{W} := \{ \rho(t_x) W \}_{x\in X}$ is a tight fusion frame whose automorphism group contains $\sigma(G) \leq S_X$.

Furthermore, if $G$ acts doubly homogeneously on $X$ then the following hold:
\begin{itemize}
\item[(a)]
$\mathcal{W}$ is an ECTFF, and
\item[(b)]
when $|X| \geq 2$, $\mathcal{W}$ is an EITFF if and only if there exists $g \in G \setminus H$ for which $W$ and $\rho(g)W$ are isoclinic.
\end{itemize}
\end{theorem}

In Theorem~\ref{thm: one layer general group}, the subspace ensemble $\mathcal{W}$ does not depend on the choice of transversal $\{ t_x \}_{x\in X}$ since $W$ is held invariant by the stabilizer $H$ of $x_0$.
Indeed, if $t_x' \in G$ also satisfies $\sigma(t_x')x_0 = x$, then $t_x^{-1} t_x' \in H$ and $\rho(t_x') W = \rho(t_x) \rho(t_x^{-1} t_x') W = \rho(t_x) W$.
In the case where $G$ has an abelian subgroup $A\leq G$ that acts regularly on $X$, then selecting $A$ as a transversal shows that $\mathcal{W}$ is \textit{harmonic}, in the sense of~\cite{FIJM:23}.

Our construction of tight fusion frames in Theorem~\ref{thm: one layer general group} generalizes the well-known construction of unit-norm tight frames by irreducible representations, as explained in Theorem~10.5 of~\cite{W:18} (for instance).
An even more general version appears in Theorem~\ref{thm: multiple layers general group} below.
The precise construction in Theorem~\ref{thm: one layer general group} was suggested as a promising method for arranging subspaces by Conway, Hardin, and Sloane in~\cite{CHS:96}, but the terminology of tight fusion frames had not yet been invented, and the authors of~\cite{CHS:96} did not provide certificates for optimality of this construction.
Whereas~\cite{C:18} essentially contains part (a) of Theorem~\ref{thm: one layer general group}, it does not contain the observation in part (b) or any of its consequences.
As far as the authors know, the content of Theorem~\ref{thm: one layer general group} is entirely new in the case where $\dim W > 1$ and $G$ does not act doubly transitively on $X$.
In the sequel, we will mainly be interested in collecting EITFFs which arise from Theorem~\ref{thm: one layer general group}(b).

\begin{proof}[Proof of Theorem~\ref{thm: one layer general group}]
Let $P \leq \mathbb{F}^{d\times d}$ be the matrix of orthogonal projection onto $W$, so that for each $x \in X$ $P_x := \rho(t_x) P \rho(t_x)^*$ is the orthogonal projection onto $\rho(t_x) W$.
Then $S:= \sum_{x \in X} P_x$ is the fusion frame operator of $\mathcal{W}$.
We claim it commutes with $\rho(g)$ for every $g \in G$.
To prove this, we first demonstrate that 
\begin{equation}
\label{eq: action on subspaces}
\rho(t_{\sigma(g)x}) W = \rho(gt_x) W
\end{equation} 
for every $x \in X$.
Indeed, the identity $\sigma( g t_x) x_0= \sigma(g) x = \sigma(t_{\sigma(g) x}) x_0$
implies $t_{\sigma(g) x}^{-1} g t_x \in H$, and since $H$ holds $W$ invariant it follows that $\rho(t_{\sigma(g)x}^{-1}gt_x) W = W$, as in~\eqref{eq: action on subspaces}.
In particular, \eqref{eq: action on subspaces} implies that
\[
P_{\sigma(g) x}
= \rho(t_{\sigma(g) x}) P \rho(t_{\sigma(g) x})^*
= \rho(g t_x) P \rho(g t_x)^*
= \rho(g) P_x \rho(g)^*,
\]
so that
\[
\rho(g) S \rho(g)^* = \sum_{x\in X} \rho(g) P_x \rho(g)^* = \sum_{x\in X} P_{\sigma(g) x} = S.
\]
This proves the claim.

Now choose an eigenvalue $\lambda \in \mathbb{F}$ of $S$.
(It exists even when $\mathbb{F} = \mathbb{R}$ since $S = S^*$.)
Applying the claim, we find that the $\lambda$-eigenspace $V\leq \mathbb{F}^d$ is invariant under $\rho(g)$ for every $g \in G$: any $v \in V$ satisfies $S \rho(g) v = \rho(g) S v = \lambda \rho(g) v$.
Since $\rho$ is irreducible, it follows that $V = \mathbb{F}^d$ and $S = \lambda I$.
In other words, $\mathcal{W}$ is a tight fusion frame.
Its automorphism group contains $\sigma(G)$ since the identity $\rho(g) P_x \rho(g)^* = P_{\sigma(g) x}$ implies $\sigma(g) \in \operatorname{Aut}(\mathcal{W})$ for each $g \in G$.

Finally, to prove the ``furthermore'' part, assume $G$ acts doubly homogeneously on $X$.
Given any two unordered pairs $\{x,y\}, \{x',y'\} \subseteq X$ with $x \neq y$ and $x' \neq y$, there exists $g \in G$ with $\{ \sigma(g) x, \sigma(g) y \} = \{ x', y' \}$.
Applying~\eqref{eq: action on subspaces}, we find that 
\[
\big\{ \rho(t_{x'}) W, \rho(t_{y'}) W \big\} 
= \big\{ \rho(t_{\sigma(g) x}) W, \rho(t_{\sigma(g) y}) W \big\}
= \big\{ \rho(g) \rho(t_x) W, \rho(g) \rho(t_y) W \big\}.
\]
In particular, $\operatorname{dist}_c\big( \rho(t_{x'}) W, \rho(t_{y'}) W \big) = \operatorname{dist}_c\big( \rho(t_x) W, \rho(t_y) W \big)$.
This proves (a).
The ``if'' part of (b) follows similarly, while the ``only if'' part holds trivially.
\end{proof}

\begin{remark}
When applying Theorem~\ref{thm: one layer general group}, there could be infinitely many candidate subspaces $W$ that are invariant under $\rho(H)$, and different choices of $W$ may lead to fusion frames with very different properties.
(This occurs, for instance, in the context of \textit{Zauner's conjecture}~\cite{Z:99,RBSC:04}, where the main problem is to choose a line from a particular eigenspace so that its orbit under the Schr\"odinger representation is equiangular.
This construction may be seen as an instance of Theorem~\ref{thm: one layer general group}, where $G$ is a certain group containing the Heisenberg group as a normal subgroup of index~$3$.)
Among this possible multitude of choices for $W$, we focus in this paper on a special class of candidates identified by Creignou~\cite{C:18}.
Let $G$, $H$, and $\rho$ be as in Theorem~\ref{thm: one layer general group}.
For each irreducible representation of $H$, the corresponding isotypic component of $\rho \big|_H$ is invariant under $\rho(H)$, and so any sum of isotypic components satisfies the condition for $W$ in Theorem~\ref{thm: one layer general group}.
In particular, when $G$ is doubly transitive on $G/H$, taking $W$ to be any sum of isotypic components of $\rho \big|_H$ produces an ECTFF via Theorem~\ref{thm: one layer general group}(a).
This method creates many examples of ECTFFs~\cite{C:18}, and they have not been carefully catalogued for $\dim W >1$.
(The case $\dim W =1$ amounts to doubly transitive lines, and the interesting examples have been classified~\cite{DK:23,IM:DTII,IM:DTI}.)
As we will see, this method also produces many EITFFs via Theorem~\ref{thm: one layer general group}(b).
\end{remark}

\subsection{Examples}
\label{subsec: exs}
Details of the following examples can be verified with software such as GAP~\cite{gap}.
Table~\ref{tbl: sporadic params} collects EITFF parameters obtained in this subsection.
We begin with a simple example that illustrates the notation of Theorem~\ref{thm: one layer general group}.

\begin{table}
\begin{center}
\begin{tabular}{c|ccccccc}
$\mathbb{F}$ & $\mathbb{R}$ & $\mathbb{R}$ & $\mathbb{R}$ & ${\mathbb{C}}$ & ${\mathbb{C}}$ & $\mathbb{R}$ & ${\mathbb{C}}$  \\
$d$ & 2 & 5 & 8 & {10} & {10} & 11 & {16}  \\
$r$ & 1 & 2 & 3 & {3} & {4} & 3 & {5}  \\
$n$ & 3 & 5 & 6 & {15} & {10} & 11 & {12} \ \\[3 pt]
$\alpha$ & $\frac{1}{2}$ & $\frac{1}{2}$ & $\frac{1}{2}$ & $\frac{1}{2}$ & $\frac{1}{\sqrt{3}}$ & $\frac{1}{\sqrt{5}}$ & $\frac{1}{2}$ \\[3 pt]
$G$ & $S_3$ & $S_5$ & $A_6$ & $A_7$ & $\operatorname{PSL}(2,9)$ & $\operatorname{PSL}(2,11)$ & $M_{11}$
\end{tabular}
\end{center}
\smallskip
\caption{
For each column above, Theorem~\ref{thm: one layer general group} constructs an $\operatorname{EITFF}_{\mathbb{F}}(d,r,n)$ with isoclinism parameter $\alpha$.
Furthermore, the automorphism group of this EITFF is doubly transitive and contains a copy of $G$.
These examples appear in Subsection~\ref{subsec: exs}.
}
\label{tbl: sporadic params}
\end{table}

\begin{example}
Theorem~\ref{thm: one layer general group} constructs the three equiangular lines in $\mathbb{R}^2$ spanned by vectors of the \textit{Mercedes--Benz frame} as follows.
Let $G := S_3$ act on $X := [3]$, and take $x_0 := 3$ for a base point.
Its stabilizer is $H = \langle (1\, 2) \rangle$, and a transversal consists of $t_1 := (1\, 2\, 3)$, $t_2 := (1\, 3\, 2)$, and $t_3 := ()$.
For an irreducible representation, take the homomorphism $\rho \colon S_3 \to \operatorname{U}(\mathbb{R}^2)$ uniquely determined by the generator images
\[
\rho(1\, 2\, 3) = \tfrac{1}{2} \left[ \begin{array}{rr} -1 & -\sqrt{3} \\ \sqrt{3} & -1 \end{array} \right]
\quad
\text{and}
\quad
\rho(1\, 2) = \left[ \begin{array}{rr} 1 & 0 \\ 0 & -1 \end{array} \right].
\]
Then the line $W:=\operatorname{span}\{ (0,1) \} \leq \mathbb{R}^2$ is held invariant by $\rho(H)$, and its orbit $\mathcal{W} = \{ \rho(t_k) W \}_{k \in [3]}$ consists of the three lines spanned by the Mercedes--Benz frame vectors $\tfrac{1}{2} (-\sqrt{3}, -1)$, $\tfrac{1}{2} (\sqrt{3}, -1)$, and $(0,1)$.
Theorem~\ref{thm: one layer general group} implies $\mathcal{W}$ is a tight fusion frame with automorphism group $S_3$.
Furthermore, since $S_3$ is doubly homogeneous, and since any two lines are isoclinic, Theorem~\ref{thm: one layer general group}(b) implies $\mathcal{W}$ is an $\operatorname{EITFF}$ (as is well known).
\end{example}

\begin{example}
\label{ex: 5 2 5}
Next, we construct a totally symmetric $\operatorname{EITFF}_{\mathbb{R}}(5,2,5)$.
This example is the smallest instance of our main results in Section~\ref{sec: total symmetry one layer}.
Consider $G := S_5$ acting on $X := [5]$, and take $x_0 := 5$ for a base point.
Its stabilizer is $H := \langle (1\, 2),\, (2\, 3),\, (3\, 4) \rangle \cong S_4$, and a transversal $\{ t_k \}_{k\in [5]}$ is given by $t_k:=(1\, 2\, 3\, 4\, 5)^k$.
There is an irreducible representation $\rho \colon S_5 \to \operatorname{U}(\mathbb{R}^5)$ uniquely determined by the generator images
\[
\rho(1\, 2) =
\left[\begin{array}{rrrrr}
1 & 0 & 0 & 0 & 0 \\
0 & 1 & 0 & 0 & 0 \\
0 & 0 & -1 & 0 & 0 \\
0 & 0 & 0 & 1 & 0 \\
0 & 0 & 0 & 0 & -1
\end{array}\right],
\quad
\rho(2\, 3) =
\tfrac{1}{2} \left[\begin{array}{rrrrr}
2 & 0 & 0 & 0 & 0 \\
0 & -1 & \sqrt{3} & 0 & 0 \\
0 & \sqrt{3} & 1 & 0 & 0 \\
0 & 0 & 0 & -1 & \sqrt{3} \\
0 & 0 & 0 & \sqrt{3} & 1
\end{array}\right],
\]
\[
\rho(3\, 4) =
\tfrac{1}{3} \left[\begin{array}{rrrrr}
-1 & 2 \, \sqrt{2} & 0 & 0 & 0 \\
2 \, \sqrt{2} & 1 & 0 & 0 & 0 \\
0 & 0 & 3 & 0 & 0 \\
0 & 0 & 0 & 3 & 0 \\
0 & 0 & 0 & 0 & -3
\end{array}\right],
\quad
\rho(4\, 5) =
\tfrac{1}{2}
\left[ \begin{array}{rrrrr}
2 & 0 & 0 & 0 & 0 \\
0 & -1 & 0 & \sqrt{3} & 0 \\
0 & 0 & -1 & 0 & \sqrt{3} \\
0 & \sqrt{3} & 0 & 1 & 0 \\
0 & 0 & \sqrt{3} & 0 & 1
\end{array}\right].
\]
Then the image $W \leq \mathbb{R}^5$ of the isometry $\Psi:=\left[ \begin{array}{rr} e_4 & e_5 \end{array} \right] \in \mathbb{R}^{5 \times 2}$ is invariant under $\rho(H)$, where $e_k$ is the $k$th column of the identity matrix.
By Theorem~\ref{thm: one layer general group}, $\mathcal{W}=\{ \rho(t_k) W \}_{k \in [5]}$ is a tight fusion frame with automorphism group $S_5$.
An explicit fusion synthesis matrix for $\mathcal{W}$ is
\[
\left[ \begin{array}{ccccc}
\rho(t_1) \Psi & \rho(t_2) \Psi & \rho(t_3) \Psi & \rho(t_4) \Psi & \rho(t_5) \Psi
\end{array} \right]
\]
\[
=
\tfrac{1}{12}
\left[ \begin{array}{rr|rr|rr|rr|rr}
4 \, \sqrt{6} & 0 & -2 \, \sqrt{6} & -6 \, \sqrt{2} & 4 \, \sqrt{6} & 0 & 0 & 0 & 0 & 0 \\
-\sqrt{3} & 9 & -4 \, \sqrt{3} & 6 & 2 \, \sqrt{3} & 0 & -3 \, \sqrt{3} & -9 & 0 & 0 \\
-3 & -3 \, \sqrt{3} & -6 & 0 & 0 & 6 \, \sqrt{3} & -9 & 3 \, \sqrt{3} & 0 & 0 \\
-3 & -3 \, \sqrt{3} & 6 & 0 & 6 & 0 & -3 & -3 \, \sqrt{3} & 12 & 0 \\
-3 \, \sqrt{3} & 3 & 0 & -6 & 0 & -6 & -3 \, \sqrt{3} & 3 & 0 & 12
\end{array} \right].
\]
Furthermore, $W$ and $\rho(4\, 5)W$ are isoclinic since $\Psi^\top \rho(4\, 5) \Psi = \tfrac{1}{2} I_2$, and so Theorem~\ref{thm: one layer general group}(b) implies $\mathcal{W}$ is a totally symmetric $\operatorname{EITFF}(5,2,5)$.
Such an ensemble was previously constructed by Et-Taoui using a different technique~\cite{ET:06}.
In Section~\ref{sec: total symmetry one layer}, we extend the present construction to an infinite family of totally symmetric EITFFs.
See Example~\ref{ex: 5 5 2 redux}
for more detail.
\end{example}

\begin{example}
We construct an $\operatorname{EITFF}_{\mathbb{R}}(8,3,6)$ from the action of $G:=A_6$ on $X:=[6]$.
Fix a base point $x_0 \in [6]$, let $H \leq G$ be its stabilizer, and choose a transversal $\{t_k\}_{k\in [6]}$.
Next, let $\rho \colon A_6 \to \operatorname{U}(\mathbb{R}^8)$ be either of the two irreducible representations with degree~8.
(Both can be taken to be real, and either one works.)
The restriction of $\rho$ to $H$ decomposes as a direct sum of two irreducible subrepresentations, one of degree~3 and another of degree~5.
Let $W \leq \mathbb{R}^8$ be the isotypic component for the 3-dimensional constituent. 
By Theorem~\ref{thm: one layer general group}, its orbit $\{ \rho(t_k)W \}_{k \in [6]}$ is a tight fusion frame whose automorphism group contains $A_6$.
This turns out to be an $\operatorname{EITFF}_{\mathbb{R}}(8,3,6)$.
EITFFs with such parameters previously appeared in~\cite{ETR:09,IKM:21}.
We include a fusion synthesis matrix with the arXiv version of this paper.
\end{example}

\begin{example}
The group $\operatorname{PSL}(2,9)$ acts doubly transitively on $10$~points~\cite{DM:96}.
Here we construct an $\operatorname{EITFF}_{\mathbb{C}}(10,4,10)$ whose automorphism group contains an isomorphic copy of this action.
Consider $G:=\operatorname{SL}(2,9) \allowbreak = \langle g_1,g_2,g_3 \rangle$, where
\[
g_1:= \left[ \begin{array}{rr} a & 0 \\ 0 & a^{-1} \end{array} \right],
\quad
g_2:=\left[ \begin{array}{rr} 1 & 1 \\ 0 & 1 \end{array} \right],
\quad
g_3:=\left[ \begin{array}{rr} 0 & -1 \\ 1 & 0 \end{array} \right],
\]
and where $a$ is a generator of the multiplicative group of the finite field $\mathbb{F}_9$.
Then $G$ acts doubly transitively on the set $X$ of one-dimensional lines through the origin of $\mathbb{F}_9^2$, and the stabilizer of the line spanned by 
$[\ 1\ 0\ ]^\top$
is the upper-triangular subgroup $H := \langle g_1,g_2 \rangle$.
This action of $G$ is not faithful, and the corresponding image in $S_{X}$ is isomorphic to $\operatorname{PSL}(2,9)$.

For a unitary representation of $G$, first let $\xi \colon G \to \mathbb{T} = \operatorname{U}(\mathbb{C})$ be either one of the two degree-1 representations of $H$ given by  $\big(\xi(g_1),\xi(g_2)\big) = \big(e^{\frac{2\pi i}{8}},1\big)$ or $\big(\xi(g_1),\xi(g_2)\big) = \big(e^{\frac{7\cdot 2\pi i}{8}},1\big)$, and then let
$\rho\colon G \to \operatorname{U}(\mathbb{C}^{10})$ be the unitary representation of $G$ \textit{induced} from $\xi$; see~\cite{CSST:18}.
It turns out that $\rho$ is irreducible, and its restriction to $H$ decomposes as a sum of pairwise inequivalent irreducible constituents with degrees $1,1,4,4$.
Let $W \leq \mathbb{C}^{10}$ be either of the $4$-dimensional isotypic components.
By Theorem~\ref{thm: one layer general group}, the orbit $\{ \rho(g) W \}_{gH \in G/H}$ of $W$ is a tight fusion frame whose automorphism group contains $\operatorname{PSL}(2,9) \leq S_{X}$.
This fusion frame turns out to be an $\operatorname{EITFF}_{\mathbb{C}}(10,4,10)$.
As far as the authors know, these parameters are new.
We include a fusion synthesis matrix as an ancillary file with the arXiv version of this paper.
\end{example}

\begin{example}
There is a sporadic doubly transitive action of $\operatorname{PSL}(2,11)$ on $11$ points; see~\cite{DM:96}.
Here we construct an $\operatorname{EITFF}_{\mathbb{R}}(11,3,11)$ from this action.
Consider $G := \operatorname{SL}(2,11) = \langle g_1,g_2,g_3 \rangle$, where
\[
g_1:= \left[ \begin{array}{rr} 6 & 10 \\ 2 & 9 \end{array} \right],
\qquad
g_2:= \left[ \begin{array}{rr} 4 & 4 \\ 4 & 7 \end{array} \right],
\qquad
g_3:= \left[ \begin{array}{rr} 0 & 1 \\ 10 & 0 \end{array} \right].
\]
Then (one can verify) $G$ has a 2-transitive action on $X:=[11]$, where a base point $x_0:=11$ is stabilized by $H := \langle g_1,g_2 \rangle$, and where a transversal $\{ t_k \}_{k \in [11]}$ is given by $t_k:= \left[ \begin{array}{rr} 9 & 4 \\ 6 & 4 \end{array} \right]^k$.
This is not a faithful action of $G$, and the corresponding image in $S_{11}$ is isomorphic to $\operatorname{PSL}(2,11)$.
Let $\rho \colon G \to \operatorname{U}(\mathbb{R}^{11})$ be an irreducible representation of degree~11.
(There is only one up to equivalence, and it can be taken to be real.)
Then the restriction of $\rho$ to $H$ decomposes as a direct sum of three inequivalent irreducible subspaces with dimensions $3,3,7$.
Let $W \leq \mathbb{R}^{11}$ be either one of the two 3-dimensional isotypic components.
Then its orbit $\mathcal{W}:=\{ \rho(t_k) W \}_{k \in [11]}$ is a tight fusion frame whose automorphism group contains a copy of $\operatorname{PSL}(2,11)$, by Theorem~\ref{thm: one layer general group}.
This fusion frame turns out to be an $\operatorname{EITFF}_{\mathbb{R}}(11,3,11)$.
Furthermore, we constructed $\mathcal{W}$ as an orbit of the regular abelian subgroup $\langle t_1 \rangle \cong \mathbb{Z}_{11}$, and so it is a \emph{harmonic} $\operatorname{EITFF}_{\mathbb{R}}(11,3,11)$ in the sense of~\cite{FIJM:23}.
Such an EITFF appeared in Theorem~4.3 of~\cite{FIJM:23}.
\end{example}

\begin{example}
Next, we construct an $\operatorname{EITFF}_{\mathbb{C}}(16,5,12)$ whose automorphism group is triply transitive and contains a copy of the Mathieu group $M_{11}$.
Let $G = M_{11}$, and let $H \leq G$ be a subgroup of index $[G : H] = 12$.
(There is only one conjugacy class of such subgroups.)
Take $\rho \colon G \to \operatorname{U}(\mathbb{C}^{16})$ to be either one of the two (complex conjugate) irreducible representations of $G$ having degree~16.
Then the restriction of $\rho$ to $H$ decomposes as a direct sum of two irreducible subspaces with dimensions $5$ and $11$.
If $W \leq \mathbb{C}^{16}$ is the 5-dimensional isotypic component, then its orbit $\{ \rho(g) W \}_{gH \in G/H}$ is a tight fusion frame whose automorphism group contains a copy of $M_{11}$, by Theorem~\ref{thm: one layer general group}.
This fusion frame turns out to be an $\operatorname{EITFF}_{\mathbb{C}}(16,5,12)$.
As far as the authors know, these parameters are new.
We include a fusion synthesis matrix as an ancillary file with the arXiv version of this paper.
\end{example}

\begin{example}
Finally, we construct an $\operatorname{EITFF}_{\mathbb{C}}(10,3,15)$ whose automorphism group is doubly transitive and contains a copy of $A_7$.
Let $G = A_7$, and let $H \leq G$ be a subgroup of index $[G : H] = 15$.
(There are two conjugacy classes, and either one works.)
Take $\rho\colon G \to \operatorname{U}(\mathbb{C}^{10})$ to be either one of the two (complex conjugate) irreducible representations of $G$ having degree~10.
Then the restriction of $\rho$ to $H$ decomposes as a direct sum of two irreducible subspaces with dimensions $3$ and $7$.
If $W \leq \mathbb{C}^{10}$ is the 3-dimensional isotypic component, then its orbit $\{ \rho(g) W \}_{gH \in G/H}$ is a tight fusion frame whose automorphism group contains an action of $A_7$ on 15~points, by Theorem~\ref{thm: one layer general group}.
This fusion frame turns out to be an $\operatorname{EITFF}_{\mathbb{C}}(10,3,15)$.
As far as the authors know, these parameters are also new.
We include a fusion synthesis matrix as an ancillary file with the arXiv version of this paper.
\end{example}

\subsection{Multiple layers}

The remainder of the paper is devoted to generalizing the totally symmetric $\operatorname{EITFF}_{\mathbb{R}}(5,2,5)$ from Example~\ref{ex: 5 2 5}.
We will utilize the construction of Theorem~\ref{thm: one layer general group}, as well as the following generalization, which produces tight fusion frames from multiplicity-free representations.

\begin{theorem}
\label{thm: multiple layers general group}
Let $G$ be a finite group acting transitively on a set $X$ via $\sigma \colon G \to S_X$.
Fix a base point $x_0 \in X$ with stabilizer $H \leq G$, and select a transversal $\{ t_x \}_{x\in X}$ of elements in $G$ such that $\sigma(t_x) x_0 = x$ for each $x \in X$.
Next, choose pairwise inequivalent, irreducible representations $\rho_i \colon G \to \operatorname{U}(\mathbb{F}^{d_i})$, $i \in [\ell]$, over a common base field $\mathbb{F} \in \{ \mathbb{R}, \mathbb{C} \}$.
Assume their restrictions to $H$ have a common constituent $\pi \colon H \to \operatorname{U}(\mathbb{F}^r)$ (not necessarily irreducible), and select isometries $\Psi_i \in \mathbb{F}^{d_i \times r}$ with the property that $\rho_i(h) \Psi_i = \Psi_i \pi(h)$ for every $i\in [\ell]$ and $h \in H$.
Finally, denote $d:=\sum_{i\in [\ell]} d_i$ and $\rho := \bigoplus_{i\in [\ell]} \rho_i \colon G \to \operatorname{U}(\mathbb{F}^d)$.
Then the vertical concatenation 
\[ \left[ \begin{array}{cccc} \sqrt{d_1/d} \Psi_1 ; & \sqrt{d_2/d} \Psi_2 ; & \cdots \ \ ; & \sqrt{d_\ell/d} \Psi_\ell \end{array} \right] \in \mathbb{F}^{d \times r} \]
is an isometry onto a subspace $W \leq \mathbb{F}^d$, and its orbit $\mathcal{W}:= \{ \rho(t_x) W \}_{x\in X}$ is a tight fusion frame whose automorphism group contains $\sigma(G) \leq S_X$.

Furthermore, if $G$ acts doubly homogeneously on $X$ then the following hold:
\begin{itemize}
\item[(a)]
$\mathcal{W}$ is an ECTFF, and
\item[(b)]
when $|X| \geq 2$, $\mathcal{W}$ is an EITFF if and only if there exists $g \in G \setminus H$ for which $W$ and $\rho(g)W$ are isoclinic.
\end{itemize}
\end{theorem}

As with Theorem~\ref{thm: one layer general group}, the subspace ensemble $\mathcal{W}$ in Theorem~\ref{thm: multiple layers general group} does not depend on the choice of transversal $\{t_x\}_{x\in X}$.
We refer to $\ell$ as the number of \textbf{layers} in Theorem~\ref{thm: multiple layers general group}, since the resulting fusion synthesis matrix decomposes as an $[\ell] \times X$ array of blocks, where the $(i,x)$ block is $\sqrt{d_i/d} \rho_i(t_x) \Psi_i \in \mathbb{F}^{d_i \times r}$.
This decomposition appears prominently in the following proof.

\begin{proof}[Proof of Theorem~\ref{thm: multiple layers general group}]
It is easy to verify that the matrix in the theorem statement is an isometry.
Let $W$ be its image, and let $P \in \mathbb{F}^{d\times d}$ be the matrix of orthogonal projection onto $W$, so that for each $x \in X$ the orthogonal projection onto $\rho(t_x) W$ is given by $P_x := \rho(t_x) P \rho(t_x)^*$.
We may view each $P_x$ as an $\ell \times \ell$ array of blocks, where the $(i,j)$ block is 
\[ (P_x)_{ij} = \frac{\sqrt{d_i d_j}}{d} \rho_i(t_x) \Psi_i \Psi_j^* \rho_j(t_x)^* \in \mathbb{F}^{d_i \times d_j}. \]
Then the fusion frame operator $S = \sum_{x\in X} P_x$ has a corresponding block structure, where
\[
S_{ij} = \frac{\sqrt{d_i d_j}}{d} \sum_{x\in X} \rho_i(t_x) \Psi_i \Psi_j^* \rho_j(t_x)^* \in \mathbb{F}^{d_i \times d_j}.
\]
We claim that $\rho_i(g) S_{ij} = S_{ij} \rho_j(g)$ for every $i,j \in [\ell]$ and $g \in G$.
To see this, choose any $g \in G$ and $x \in X$.
As in the proof of Theorem~\ref{thm: one layer general group}, it holds that $h:= t_{\sigma(g) x}^{-1} g t_x \in H$, and so 
\[
\rho_k(g) \rho_k( t_x) \Psi_k = \rho_k( t_{\sigma(g) x}) \rho_k( t_{\sigma(g) x}^{-1} g t_x) \Psi_k = \rho_k( t_{\sigma(g) x} ) \Psi_k \pi(h)
\]
for every $k \in [\ell]$.
Given any $i,j \in [\ell]$, we deduce the crucial identity
\begin{equation}
\label{eq: intertwine}
\rho_i(g) \rho_i( t_x) \Psi_i \Psi_j^* \rho_j(t_x)^* \rho_j(g)^*
= \rho_i( t_{\sigma(g) x} ) \Psi_i \Psi_j^* \rho_j( t_{\sigma(g) x} )^*.
\end{equation}
Summing over $x \in X$ and then re-indexing, we quickly deduce that $\rho_i(g) S_{ij} \rho_j(g)^* = S_{ij}$, as claimed.

For $i \neq j$, the irreducible representations $\rho_i$ and $\rho_j$ are inequivalent, and our claim together with Schur's Lemma implies $S_{ij} = 0$.
Similarly, for $i = j$, we have $S_{ii}^* = S_{ii}$, and it follows as in the proof of Theorem~\ref{thm: one layer general group} that $S_{ii} = c_i I_{d_i}$ for some constant $c_i$.
The constant $c_i$ does not depend on $i$, as seen by taking a trace:
\[
c_i d_i = \operatorname{tr}(c_i I_{d_i})
= \operatorname{tr}(S_{ii})
= \frac{d_i}{d} \sum_{x\in X} \operatorname{tr}\Bigl[ \rho_i(t_x) \Psi_i^* \Psi_i \rho_i(t_x)^* \Bigr]
= \frac{d_i}{d} \sum_{x\in X} \operatorname{tr}( I_r )
= \frac{d_i}{d} r |X|.
\]
Hence, $S = \frac{r |X|}{d} I_d$, and $\mathcal{W}$ is a tight fusion frame.
To see that its automorphism group contains $\sigma(G)$, we may interpret~\eqref{eq: intertwine} to say $\rho_i(g) (P_x)_{ij} \rho_j(g)^* = (P_{\sigma(g) x})_{ij}$, so that $\rho(g) P_x \rho(g)^* = P_{\sigma(g) x}$ for every $g\in G$ and $x \in X$.

Finally, the proof of the ``furthermore'' part is similar to that in Theorem~\ref{thm: one layer general group}.
\end{proof}

\section{Review of representation theory of $S_n$}
\label{sec: Sn review}

In order to produce more EITFFs with the techniques from Section~\ref{sec: tight fusion frames from symmetry}, we now recall some representation theory of the symmetric group $S_n$~\cite{CSST:10}.
Some of the notation that follows is nonstandard, but it will prove helpful in the sequel.

As detailed below, each irreducible representation of $S_n$ is associated with a unique \textbf{partition} $\lambda \vdash n$, that is, a nonincreasing sequence of positive integers that sum to~$n$.
We use superscripts to indicate repeated entries in this sequence:
\[
\lambda = (\underbrace{l_1,\dotsc,l_1}_{a_1 \text{ copies }},\cdots,\underbrace{l_m,\dotsc,l_m}_{a_m \text{ copies }}) =: (l_1^{a_1},\dotsc,l_m^{a_m}).
\]
Here, if $l_1 > \dotsb > l_m$, then we say $\lambda$ has $m$ \textbf{distinct parts}.

A \textbf{Young diagram} depicts a partition $\lambda = (\lambda_1,\dotsc,\lambda_h)$ as an array of $n$ boxes arranged on a grid, where the $i$th row from the top contains $\lambda_i$ boxes.
For example, Figure~\ref{fig: young ex}(a) shows the Young diagram with shape $(4,2^2) \vdash 8$.
Boxes of a Young diagram are indexed just like matrix entries: rows from top to bottom and columns from left to right.
For $c \in \mathbb{Z}$, the $c$th \textbf{superdiagonal} of a Young diagram consists of all boxes in position $(i,j)$ with $j-i=c$.
Thus, the main diagonal is the $0$th superdiagonal, with positive superdiagonals above it and negative ones below it.
In Figure~\ref{fig: young ex}(a), the upper blue box is on the 3rd superdiagonal, and the lower one is on the $-1$st.  
The \textbf{axial distance} from the $(i,j)$ box to the $(k,l)$ box is defined as the difference of their superdiagonals: 
\[
D\Bigl( (i,j),\, (k,l) \Bigl) := (j-i) - (l-k).
\]
This is the total signed distance of any path from $(i,j)$ to $(k,l)$, where paths that go down or to the left count as having positive length, and those that go up or to the right count negative.
In Figure~\ref{fig: young ex}(a), the axial distance from the lower blue square to the upper one is $-4$.

\begin{figure}
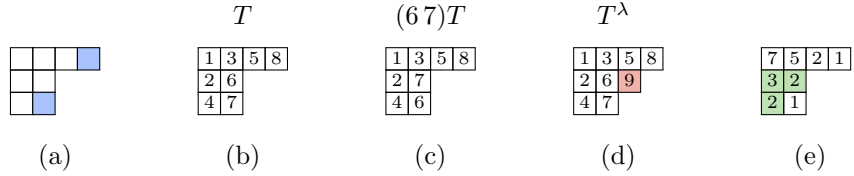

\begin{center}
\ytableausetup{smalltableaux}
\begin{tabular}{ccccc}
\centering
& $T$ & $(6\, 7) T$ & $T^{\lambda}$ \vspace{5 pt} \\
\hspace{10 pt}
\begin{ytableau}
~ & ~ & ~ & *(ltblue) ~  \\
~ & ~ \\
~ & *(ltblue) ~
\end{ytableau}
\hspace{10 pt}
&\hspace{10 pt}
\begin{ytableau}
1 & 3 & 5 & 8  \\
2 & 6 \\
4 & 7
\end{ytableau}
\hspace{10 pt}
&
\hspace{10 pt}
\begin{ytableau}
1 & 3 & 5 & 8  \\
2 & 7 \\
4 & 6
\end{ytableau}
\hspace{10 pt}
&
\hspace{10 pt}
\begin{ytableau}
1 & 3 & 5 & 8  \\
2 & 6 & *(ltpink) 9 \\
4 & 7
\end{ytableau}
\hspace{10 pt}
&
\hspace{10 pt}
\begin{ytableau}
7 & 5 & 2 & 1  \\
*(ltgreen) 3 & *(ltgreen) 2 \\
*(ltgreen) 2 & 1
\end{ytableau}
\hspace{10 pt}
\vspace{10 pt} \\
(a) & (b) & (c) & (d) & (e)
\end{tabular}
\end{center}
\caption{
\footnotesize{
(a) The Young diagram with shape $\mu := (4,2^2)$, where removable boxes are colored blue.
(b) A standard tableau $T \in \operatorname{Tab}(\mu)$.
(c) The action of a transposition on $T$.
(d) The embedding of $T$ into $\operatorname{Tab}(\lambda)$ for $\lambda := (4,3,2) \in \mu^\uparrow$, where the $\lambda-\mu$ box is colored red.
(e) Hook lengths for $\mu$, where one hook is colored green.
}
}
\label{fig: young ex}
\end{figure}

A \textbf{Young tableau} $T$ with shape $\lambda$ assigns the numbers $1,\dotsc,n$ bijectively to the boxes in the Young diagram of $\lambda$.
We write $T_{ij}$ for the number in the $(i,j)$ box of $T$.
Then $S_n$ acts on the set of all tableaux with shape $\lambda$ by the formula $(gT)_{ij} = g(T_{ij})$ for $g \in S_n$.
Figure~\ref{fig: young ex}(b) shows an example of a tableau $T$ with shape $(4,2,2)$, and then Figure~\ref{fig: young ex}(c) shows $(6\, 7) T$.
A tableau is called \textbf{standard} if its entries increase moving down and to the right: if $i < i'$ and $j < j'$ then $T_{ij} < T_{i'j}$ and $T_{ij} < T_{ij'}$.
In Figure~\ref{fig: young ex}, $T$ is standard and $(6\, 7)T$ is not.
The set of all standard tableaux with shape $\lambda \vdash n$ is denoted $\operatorname{Tab}(\lambda)$.
The \textbf{content} of a (possibly nonstandard) tableau $T$ with shape $\lambda$ is the sequence $C(T) = (a_1,\dotsc,a_n)$ of superdiagonal positions of $1,\dotsc,n$; that is, $a_k = j - i$ when $k = T_{ij}$.
The \textbf{axial distance} from $i\in [n]$ to $j \in [n]$ in $T$ is the same as for their box positions, $D_T(i,j):=a_i - a_j$.
For example, the tableau $T$ in Figure~\ref{fig: young ex}(b) has content 
$C(T) = (0,-1,1,-2,2,0,-1,3)$, and some of its axial distances are $D_T(5,4) = 4$, $D_T(1,6) = 0$, and $D_T(7,6) = -1$.

As we now explain, each partition $\lambda \vdash n$ determines a representation $\pi_\lambda$ of $S_n$ on the real Hilbert space $V_\lambda$ with formal orthonormal basis $\{ v_T \}_{T \in \operatorname{Tab}(\lambda)}$.
It suffices to describe $\pi_\lambda$ on the generating set of adjacent transpositions $s_k := (k\, k+1)$, $k \in [n-1]$, and it suffices to give the action of $\pi_\lambda(s_k)$ on the basis vectors $v_T$, $T \in \operatorname{Tab}(\lambda)$.
In this paper, we take $\pi_\lambda$ in \textbf{Young's orthogonal form}, where
\begin{equation}
\label{eq:Young's orthogonal form}
\pi_\lambda(s_k) v_T =
\frac{1}{D_T(k+1,k)} v_T + \sqrt{ 1 - \frac{1}{D_T(k+1,k)^2} } v_{s_k T}.
\end{equation}
(Above, the coefficient $\sqrt{ 1 - \frac{1}{D_T(k+1,k)^2} }$ is zero precisely when $s_k T \notin \operatorname{Tab}(\lambda)$, and in this case it does not matter that $v_{s_k T}$ is undefined.)
Expressed as a matrix with entries indexed by $\operatorname{Tab}(\lambda) \times \operatorname{Tab}(\lambda)$,
\[
[ \pi_\lambda(s_k) ]_{S,T} = 
\begin{cases}
\frac{1}{D_T(k+1,k)} & \text{if } S = T, \\[5 pt]
\sqrt{ 1 - \frac{1}{D_T(k+1,k)^2} } & \text{if } S = s_k T, \\[5 pt]
0 & \text{otherwise}.
\end{cases}
\]
Then $\pi_\lambda(s_k)$ is a unitary operator on $V_\lambda$, and the images $\{ \pi_\lambda(s_k) : k \in {[n-1]}\}$ determine an absolutely irreducible unitary representation $\pi_\lambda \colon S_n \to \operatorname{U}(V_\lambda)$.
Furthermore, this accounts for all irreducible representations of $S_n$, and when $\lambda_1 \neq \lambda_2$ the representations $\pi_{\lambda_1}$ and $\pi_{\lambda_2}$ are inequivalent.

The degree of $\pi_\lambda$ can be computed with the aid of the Young diagram for $\lambda$ as follows.
In a Young diagram, the \textbf{hook} with corner at position $(k,l)$ consists of all boxes in position $(i,j)$ where either $i = k$ and $j \geq l$, or $i \geq k$ and $j = l$.
In other words, a hook consists of a box, all boxes underneath it, and all boxes to its right.
The Young diagram in Figure~\ref{fig: young ex}(e) shows a hook highlighted in green.
The number of boxes in a hook is its \textbf{length}, denoted $h_\lambda(k,l)$ for the hook with corner at position $(k,l)$.
In Figure~\ref{fig: young ex}(e), the length of each hook in the Young diagram is listed at its corner.
The \textbf{hook length formula} states that
\[
d_\lambda := \dim(V_\lambda) = | \operatorname{Tab}(\lambda) | = \frac{n!}{\prod_{(k,l)} h_\lambda(k,l) },
\]
where the product is over all box positions in the Young diagram for $\lambda$.
For example, one can compute $d_{(4,2,2)}$ with the aid of Figure~\ref{fig: young ex}(e):
\[
d_{(4,2,2)} = \frac{8!}{7 \cdot 5 \cdot 2 \cdot 1 \cdot 3 \cdot 2 \cdot 2 \cdot 1} = 48.
\]

The \textbf{transpose} of a partition $\lambda = (\lambda_1,\dotsc,\lambda_h) \vdash n$ is defined as the partition $\lambda' = (\lambda_1',\dotsc,\lambda'_{h'}) \vdash n$ given by $\lambda_i' := \max\{ j : \lambda_j \geq i\}$, where $h' := \lambda_1$.
This terminology is appropriate since the Young diagram of $\lambda'$ is obtained by reflecting the Young diagram of $\lambda$ across its main diagonal.
Then $d_{\lambda'} = d_\lambda$ (as shown for instance by the Hook length formula).
If $\lambda = \lambda'$ then $\lambda$ is called \textbf{symmetric}.
Similarly, the \textbf{transpose} of a tableau $T$ with shape $\lambda$ is defined to be the tableau $T'$ with shape $\lambda'$ and entries $T'_{ij} = T_{ji}$.

The irreducible constituents of the restriction $\pi_\lambda \big|_{S_{n-1}}$ are necessarily indexed by certain partitions of $n-1$, and they can be identified with the aid of a Young diagram as follows.
In the Young diagram for $\lambda = (\lambda_1,\dotsc,\lambda_h)$, the $(k,l)$ box is called \textbf{removable} if $l = \lambda_k$ and either $k = h$ or $\lambda_k > \lambda_{k+1}$.
In Figure~\ref{fig: young ex}(a), the removable boxes are colored blue.
Equivalently, the $(k,l)$ box is removable if its deletion produces a valid Young diagram for another partition $\mu \vdash (n-1)$.
We express the relationship between such $\lambda$ and $\mu$ by writing
$\mu \in \lambda^\downarrow$
and 
$\lambda \in \mu^\uparrow$,
and we write $\lambda - \mu := (k,l)$ for the index of the removed box.
We use 
$\lfloor \lambda \rfloor := \{ \lambda - \mu : \mu \in \lambda^\downarrow\}$
for the set of box indices that can be removed from $\lambda$.
Notice that the number $| \lfloor \lambda \rfloor |$ of removable boxes is precisely the number of distinct entries in the sequence $\lambda$, i.e., the number of distinct parts.
An inductive argument shows that a box is removable if and only if it contains $n$ in some standard tableau:
\begin{equation}
\label{eq: removable boxes get n}
\lfloor \lambda \rfloor
= \big\{ (i,j): T_{ij} = n \text{ for some }T \in \operatorname{Tab}(\lambda) \big\}.
\end{equation}
The \textbf{branching rule} says that each constituent of $\pi_\lambda \big|_{S_{n-1}}$ occurs with multiplicity one, and it is obtained by removing a box:
\[ 
\pi_\lambda \big|_{S_{n-1}} \cong \bigoplus_{\mu \in \lambda^\downarrow} \pi_\mu.
\]

For an explicit embedding, choose $\mu \in \lambda^\downarrow$, and let $(k,l) = \lambda - \mu$ index the removed box.
Given a tableau $R$ with shape $\mu$, we define $R^\lambda$ to be the tableau that
restores the $(k,l)$ box to $R$ and fills it with the number $n$.
Figure~\ref{fig: young ex}(d) shows an example.
Explicitly, $R^\lambda_{ij} = R_{ij}$ for $(i,j) \neq (k,l)$, and $R^\lambda_{kl} = n$.
Then 
$D_R(i,j) = D_{R^\lambda}(i,j)$ for every $i,j \in [n-1]$.
Furthermore, 
$(gR)^\lambda = g (R^\lambda)$ for every $g \in S_{n-1}$, and
$R^\lambda \in \operatorname{Tab}(\lambda)$ if and only if $R \in \operatorname{Tab}(\mu)$.
Now consider the linear isometry $\Psi_{\lambda,\mu} \colon V_\mu \to V_\lambda$ given by $\Psi_{\lambda,\mu} v_R = v_{R^\lambda}$ for $R \in \operatorname{Tab}(\mu)$.
As a matrix, the entries of $\Psi_{\lambda,\mu}$ are indexed by $\operatorname{Tab}(\lambda) \times \operatorname{Tab}(\mu)$ and can be expressed in terms of the Kronecker delta:
\[
[\Psi_{\lambda,\mu}]_{T,R} = \delta_{T, R^\lambda}.
\]
Then it is easy to check that $\Psi_{\lambda,\mu} \pi_\mu(s_k) v_R = \pi_\lambda(s_k) \Psi_{\lambda,\mu} v_R$ for every $k \in [n-2]$ and $R \in \operatorname{Tab}(\mu)$, so that 
\[
\Psi_{\lambda,\mu} \pi_\mu(g) = \pi_\lambda(g) \Psi_{\lambda,\mu},
\qquad
g \in S_{n-1}.
\]
In particular, the $\pi_\mu$-isotypic component of $\pi_\lambda \big|_{S_{n-1}}$ is
\[
\operatorname{im} \Psi_{\lambda,\mu} 
= \operatorname{span}\{ v_{R^\lambda} : R \in \operatorname{Tab}(\mu) \}
= \operatorname{span}\{ v_T : T \in \operatorname{Tab}(\lambda),\, T_{kl} = n \}.
\]

\section{Total symmetry: One layer}
\label{sec: total symmetry one layer}

Having reviewed the necessary background, we now apply the single-layer construction of Theorem~\ref{thm: one layer general group} to the representations from Section~\ref{sec: Sn review}, obtaining new and totally symmetric EITFFs in Theorem~\ref{thm: single layer}.
Specifically, select a partition $\lambda \vdash n$ and an isotypic component $W \leq V_\lambda$ of $\pi_\lambda \big|_{S_{n-1}}$.
Then the action of $S_{n-1}$ on $W$ is equivalent to $\pi_\mu$ for some $\mu \in \lambda^\downarrow$.
Choose a transversal $\{t_k \}_{k\in [n]}$ in $S_n$ with $t_k n = k$ for each $k$, and consider the orbit $\mathcal{W} := \{ \pi_\lambda(t_k) W \}_{k \in [n]}$.
Then Theorem~\ref{thm: one layer general group}(a) asserts that $\mathcal{W}$ is a totally symmetric ECTFF.
As detailed below, there are many choices of $\lambda$ and $\mu$ for which $\mathcal{W}$ is an EITFF.
We begin with some illustrative examples, which we encourage the reader to study carefully.

\begin{example}
\label{ex: 5 5 2 redux}
In the context described above, take $\lambda = (3,2)$ and $\mu = (2,2)$.
Their Young diagrams are shown below, where the red box belongs to $\lambda$ but not $\mu$, and where the only removable box of $\mu$ is colored blue.
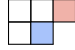
\begin{center}
\begin{ytableau}
~ & ~ & *(ltpink) ~  \\
~ &  *(ltblue) ~ \\
\end{ytableau}
\end{center}
We will show that the $\pi_\mu$-isotypic component $W = \operatorname{im} \Psi_{\lambda, \mu} \leq V_\lambda$ is isoclinic with its image $\pi_\lambda(s_4) W = \operatorname{im} \pi_\lambda(s_4) \Psi_{\lambda, \mu}$.
Then Theorem~\ref{thm: one layer general group} will imply the orbit $\mathcal{W}$ is a totally symmetric~$\operatorname{EITFF}_{\mathbb{R}}(d_\lambda,d_\mu,5)$.

By the hook length formula, there are $d_\mu = \frac{4!}{3 \cdot 2 \cdot 2 \cdot 1} = 2$ standard tableaux with shape $\mu$, and $d_\lambda = \frac{5!}{4 \cdot 3 \cdot 2 \cdot 1 \cdot 1} = 5$ standard tableaux with shape $\lambda$, all shown below.
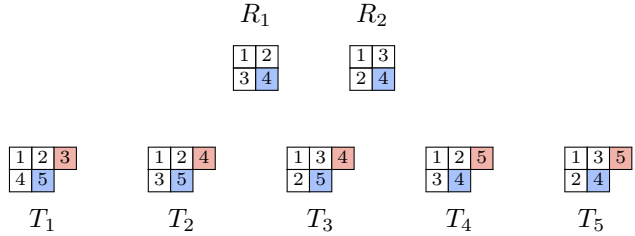
\begin{center}
\begin{tabular}{cc}
$R_1$ & $R_2$ \\[5 pt]
\hspace{5 pt}
\begin{ytableau}
1 & 2 \\
3 & *(ltblue) 4 \\
\end{ytableau}
\hspace{5 pt}
&
\hspace{5 pt}
\begin{ytableau}
1 & 3 \\
2 & *(ltblue) 4 \\
\end{ytableau}
\hspace{5 pt}
\end{tabular}

\vspace{20 pt}

\begin{tabular}{ccccc}
\hspace{5 pt}
\begin{ytableau}
1 & 2 & *(ltpink) 3 \\
4 & *(ltblue) 5 \\
\end{ytableau}
\hspace{5 pt}
&
\hspace{5 pt}
\begin{ytableau}
1 & 2 & *(ltpink) 4  \\
3 & *(ltblue) 5 \\
\end{ytableau}
\hspace{5 pt}
&
\hspace{5 pt}
\begin{ytableau}
1 & 3 & *(ltpink) 4  \\
2 & *(ltblue) 5 \\
\end{ytableau}
\hspace{5 pt}
&
\hspace{5 pt}
\begin{ytableau}
1 & 2 & *(ltpink) 5  \\
3 & *(ltblue) 4 \\
\end{ytableau}
\hspace{5 pt}
&
\hspace{5 pt}
\begin{ytableau}
1 & 3 & *(ltpink) 5  \\
2 & *(ltblue) 4 \\
\end{ytableau}
\hspace{5 pt} \\[10 pt]
$T_1$ & $T_2$ & $T_3$ & $T_4$ & $T_5$
\end{tabular}
\end{center}
To express $\Psi_{\lambda,\mu} \colon V_\mu \to V_\lambda$ as a matrix, observe that $R_1^\lambda = T_4$ and $R_2^\lambda = T_5$.
(These are exactly the $\lambda$-tableaux with $n=5$ in the red box, and they necessarily have $n-1=4$ in the blue box.)
With the bases ordered as above, we have
\[
\Psi_{\lambda,\mu} = \left[ \begin{array}{cc}
0 & 0 \\
0 & 0 \\
0 & 0 \\
1 & 0 \\
0 & 1 
\end{array} \right].
\]
Then the cross-Gram $\Psi_{\lambda,\mu}^\top \pi_\lambda(s_4) \Psi_{\lambda,\mu}$ is the bottom-right $2\times 2$ submatrix of 
\[
\pi_\lambda(s_4) =
\left[ \begin{array}{rrrrr}
1 & 0 & 0 & 0 & 0 \\[5 pt]
0 & -\tfrac{1}{2} & 0 & \tfrac{\sqrt{3}}{2} & 0 \\[5 pt]
0 & 0 & -\tfrac{1}{2} & 0 & \tfrac{\sqrt{3}}{2} \\[5 pt]
0 & \tfrac{\sqrt{3}}{2} & 0 & \cellcolor{ltgreen} \tfrac{1}{2} & \cellcolor{ltgreen} 0 \\[5 pt]
0 & 0 & \tfrac{\sqrt{3}}{2} & \cellcolor{ltgreen} 0 & \cellcolor{ltgreen} \tfrac{1}{2} \\[5 pt]
\end{array} \right],
\]
where we have applied~\eqref{eq:Young's orthogonal form}.
Explicitly, the corner highlighted in green is determined by $\pi_\lambda(s_4) v_{T_4}$ and $\pi_\lambda(s_4) v_{T_5}$.
In~\eqref{eq:Young's orthogonal form}, we have $s_4 T_4 = T_2$ and $D_{T_4}(5,4) = 2$, so that 
\[
\pi_\lambda(s_4) v_{T_4} 
= \tfrac{1}{D_{T_4}(5,4)} v_{T_4} + \sqrt{ 1 - \tfrac{1}{D_{T_4}(5,4)^2} } v_{s_4 T_4}
= \tfrac{1}{2} v_{T_4} + \tfrac{\sqrt{3}}{2} v_{T_2}.
\]
Likewise, $s_4 T_5 = T_3$ and $D_{T_5}(5,4) = 2$, so that 
\[
\pi_\lambda(s_4) v_{T_5} 
= \tfrac{1}{D_{T_5}(5,4)} v_{T_5} + \sqrt{ 1 - \tfrac{1}{D_{T_5}(5,4)^2} } v_{s_4 T_5}
= \tfrac{1}{2} v_{T_5} + \tfrac{\sqrt{3}}{2} v_{T_2}.
\]
Thus, the cross-Gram matrix is
\[
\Psi_{\lambda,\mu}^\top \pi_\lambda(s_4) \Psi_{\lambda,\mu}
=
\left[ \begin{array}{rr}
\cellcolor{ltgreen} \tfrac{1}{2} & \cellcolor{ltgreen} 0 \\[5 pt]
\cellcolor{ltgreen} 0 & \cellcolor{ltgreen} \tfrac{1}{2} \\[5 pt]
\end{array} \right],
\]
where the diagonal entries are determined by the axial distance from the red box to the blue box.
Since the cross-Gram is a multiple of a unitary, the corresponding spaces are isoclinic, and the orbit is a totally symmetric $\operatorname{EITFF}(5,2,5)$.
In fact, it is the one from Example~\ref{ex: 5 2 5}.
\end{example}

\begin{example}
\label{ex: 16 6 6}
Next, consider $\lambda = (3,2,1)$ and $\mu = (3,1,1)$.
Their Young diagrams are shown below, where the red box belongs to $\lambda$ but not $\mu$, and where the removable boxes of $\mu$ are colored blue.
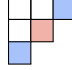
\begin{center}
\begin{ytableau}
~ & ~ & *(ltblue) ~  \\
~ &  *(ltpink) ~ \\
*(ltblue) ~
\end{ytableau}
\end{center}
Considering the following hook lengths, we have 
$d_\lambda = \frac{6!}{5\cdot 3 \cdot 3} = 16$ and $d_\mu = \frac{5!}{5 \cdot 2 \cdot 2} =6$.
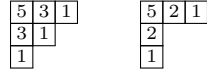
\begin{center}
\begin{ytableau}
5 & 3 & 1 \\
3 & 1 \\
1
\end{ytableau}
\qquad
\begin{ytableau}
5 & 2 & 1 \\
2 \\
1
\end{ytableau}
\end{center}
As in the last example, we will show $W:= \operatorname{im} \Psi_{\lambda, \mu} \leq V_\lambda$ is isoclinic with $\pi_\lambda(s_5) W = \operatorname{im} \pi_\lambda(s_5) \Psi_{\lambda, \mu}$, and it will follow that the orbit of $W$ under $\pi_\lambda(S_6)$ is a totally symmetric~$\operatorname{EITFF}_{\mathbb{R}}(16,6,6)$.

The standard tableaux with shape $\mu$ are shown below, along with their embeddings in~$\operatorname{Tab}(\lambda)$.
There are $10$ additional standard tableaux with shape $\lambda$ that are not displayed, but they will not matter in this example.
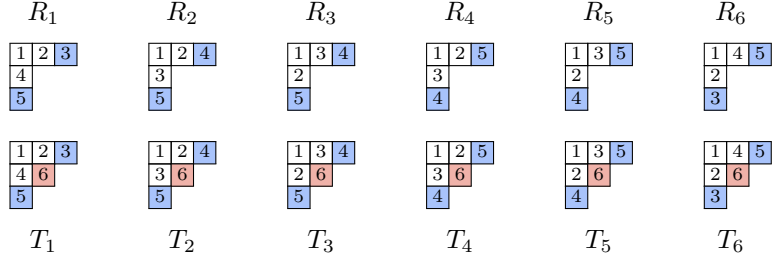
\begin{center}
\begin{tabular}{cccccc}
$R_1$ & $R_2$ & $R_3$ & $R_4$ & $R_5$ & $R_6$ \\[5 pt]
\hspace{5 pt}
\begin{ytableau}
1 & 2 & *(ltblue) 3 \\
4 \\
*(ltblue) 5
\end{ytableau}
\hspace{5 pt}
&
\hspace{5 pt}
\begin{ytableau}
1 & 2 & *(ltblue) 4 \\
3 \\
*(ltblue) 5
\end{ytableau}
\hspace{5 pt}
&
\hspace{5 pt}
\begin{ytableau}
1 & 3 & *(ltblue) 4 \\
2 \\
*(ltblue) 5
\end{ytableau}
\hspace{5 pt}
&
\hspace{5 pt}
\begin{ytableau}
1 & 2 & *(ltblue) 5 \\
3 \\
*(ltblue) 4
\end{ytableau}
\hspace{5 pt}
&
\hspace{5 pt}
\begin{ytableau}
1 & 3 & *(ltblue) 5 \\
2 \\
*(ltblue) 4
\end{ytableau}
\hspace{5 pt}
&
\hspace{5 pt}
\begin{ytableau}
1 & 4 & *(ltblue) 5 \\
2 \\
*(ltblue) 3
\end{ytableau}
\hspace{5 pt} \\[25 pt]
\hspace{5 pt}
\begin{ytableau}
1 & 2 & *(ltblue) 3 \\
4 & *(ltpink) 6 \\
*(ltblue) 5
\end{ytableau}
\hspace{5 pt}
&
\hspace{5 pt}
\begin{ytableau}
1 & 2 & *(ltblue) 4 \\
3 & *(ltpink) 6 \\
*(ltblue) 5
\end{ytableau}
\hspace{5 pt}
&
\hspace{5 pt}
\begin{ytableau}
1 & 3 & *(ltblue) 4 \\
2 & *(ltpink) 6 \\
*(ltblue) 5
\end{ytableau}
\hspace{5 pt}
&
\hspace{5 pt}
\begin{ytableau}
1 & 2 & *(ltblue) 5 \\
3 & *(ltpink) 6 \\
*(ltblue) 4
\end{ytableau}
\hspace{5 pt}
&
\hspace{5 pt}
\begin{ytableau}
1 & 3 & *(ltblue) 5 \\
2 & *(ltpink) 6 \\
*(ltblue) 4
\end{ytableau}
\hspace{5 pt}
&
\hspace{5 pt}
\begin{ytableau}
1 & 4 & *(ltblue) 5 \\
2 & *(ltpink) 6 \\
*(ltblue) 3
\end{ytableau}
\hspace{5 pt} \\[20 pt]
$T_1$ & $T_2$ & $T_3$ & $T_4$ & $T_5$ & $T_6$
\end{tabular}
\end{center}
For each $i \in [6]$, $R_i^\lambda = T_i$ and $\Psi_{\lambda,\mu} v_{R_i} = v_{T_i}$ , so the cross-Gram matrix has entries
\[
\big[ \Psi_{\lambda,\mu}^\top \pi_\lambda(s_5) \Psi_{\lambda,\mu} \big]_{ij} 
= \big \langle \Psi_{\lambda,\mu}^\top \pi_\lambda(s_5) \Psi_{\lambda,\mu} v_{R_j}, v_{R_i} \big\rangle
= \big\langle \pi_\lambda(s_5) v_{T_j}, v_{T_i} \big\rangle,
\qquad i,j \in [6].
\]
When computing $\pi_\lambda(s_5) v_{T_j}$ with~\eqref{eq:Young's orthogonal form}, notice that $s_5 T_j \notin\{T_i : i \in [6] \}$ since $[s_5 T_j]_{22} = 5$ and $[T_i]_{22} = 6$ for $i,j \in [6]$.
Applying~\eqref{eq:Young's orthogonal form}, we find
\begin{align*}
\big[ \Psi_{\lambda,\mu}^\top \pi_\lambda(s_5) \Psi_{\lambda,\mu} \big]_{ij} 
= \big\langle \pi_\lambda(s_5) v_{T_j}, v_{T_i} \big\rangle
&= \Big\langle \tfrac{1}{D_{T_j}(6,5)} v_{T_j} + \sqrt{ 1 - \tfrac{1}{D_{T_j}(6,5)^2} }v_{s_5 T_j}, v_{T_i} \Big\rangle \\
&= \frac{\delta_{ij}}{D_{T_j}(6,5)}.
\end{align*}
Thus, the cross-Gram is diagonal, and each diagonal entry is determined by an axial distance from the red box to one of the two blue boxes.
Since these two axial distances are opposites, the diagonal has constant absolute value:
\[
\Psi_{\lambda,\mu}^\top \pi_\lambda(s_5) \Psi_{\lambda,\mu}
= \left[ \begin{array}{rrrrrr}
\tfrac{1}{2} & 0 & 0 & 0 & 0 & 0 \\
0 & \tfrac{1}{2} & 0 & 0 & 0 & 0 \\
0 & 0 & \tfrac{1}{2} & 0 & 0 & 0 \\
0 & 0 & 0 & -\tfrac{1}{2} & 0 & 0 \\
0 & 0 & 0 & 0 & -\tfrac{1}{2} & 0 \\
0 & 0 & 0 & 0 & 0 & -\tfrac{1}{2}
\end{array} \right].
\]
In particular, $W$ and $\pi_\lambda(s_5)W$ are isoclinic.
By Theorem~\ref{thm: one layer general group}, the orbit of $W$ under $\pi_\lambda(S_6)$ is a totally symmetric~$\operatorname{EITFF}(16,6,6)$.
\end{example}

\begin{figure}[b]
\begin{center}
\begin{tabular}{ccc}
\hspace{10 pt}
\begin{tikzpicture}

\draw[thick] (0,0) -- (1,0) -- (1,1) -- (0,1) -- cycle;

\filldraw[thick,fill=ltblue] (1,0) rectangle (1-0.2,0+0.2);
\filldraw[thick,fill=ltpink] (1,1) rectangle (1+0.2,1-0.2);

\draw[<->] (-0.2,0+0.02) -- (-0.2,1-0.02);

\node[] at (-0.5,0.5) {$a$};

\draw[<->] (0+0.02,1+0.2) -- (1-0.02,1+0.2);

\node[] at (0.5,1.5) {$b$};

\end{tikzpicture}
\hspace{10 pt}
&
\hspace{10 pt}
\begin{tikzpicture}

\draw[thick] (0,0) -- (1,0) -- (1,1) -- (0,1) -- cycle;

\filldraw[thick,fill=ltblue] (1,0) rectangle (1-0.2,0+0.2);
\filldraw[thick,fill=ltpink] (0,0) rectangle (0+0.2,0-0.2);

\draw[<->] (-0.2,0+0.02) -- (-0.2,1-0.02);

\node[] at (-0.5,0.5) {$b$};

\draw[<->] (0+0.02,1+0.2) -- (1-0.02,1+0.2);

\node[] at (0.5,1.5) {$a$};

\end{tikzpicture}
\hspace{10 pt}
&
\hspace{10 pt}
\begin{tikzpicture}

\draw[thick] (0,0) -- (1,0) -- (1,1) -- (2,1) -- (2,2) -- (0,2) -- cycle;

\filldraw[thick,fill=ltblue] (1,0) rectangle (1-0.2,0+0.2);
\filldraw[thick,fill=ltblue] (2,1) rectangle (2-0.2,1+0.2);
\filldraw[thick,fill=ltpink] (1,1) rectangle (1+0.2,1-0.2);

\draw[<->] (-0.2,0+0.02) -- (-0.2,1-0.02);
\draw[<->] (-0.2,1+0.02) -- (-0.2,2-0.02);

\node[] at (-0.5,0.5) {$c$};
\node[] at (-0.5,1.5) {$a$};

\draw[<->] (0+0.02,2+0.2) -- (1-0.02,2+0.2);
\draw[<->] (1+0.02,2+0.2) -- (2-0.02,2+0.2);

\node[] at (0.5,2.5) {$b$};
\node[] at (1.5,2.5) {$c$};

\end{tikzpicture} 
\hspace{10 pt} \\[5 pt]
(i) & (ii) & (iii)
\end{tabular}
\end{center}

\caption{
Illustration of the Young diagrams for cases (i)--(iii) of Theorem~\ref{thm: single layer}, where the red box belongs to $\lambda$ but not $\mu$, and where the removable boxes of $\mu$ are colored blue.
}
\label{fig: single layer Young}
\end{figure}
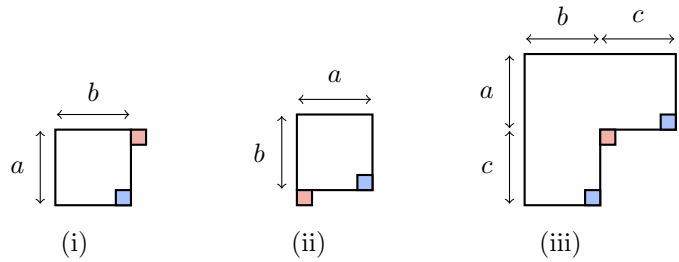

The following theorem extends Examples~\ref{ex: 5 5 2 redux} and~\ref{ex: 16 6 6} to an infinite family of totally symmetric EITFFs.

\begin{theorem}
\label{thm: single layer}
Select partitions $\lambda \vdash n$ and $\mu \in \lambda^\downarrow$, and let $W \leq V_\lambda$ be the $\pi_\mu$-isotypic component of $\pi_\lambda \big|_{S_{n-1}}$.
Assume $W \neq V_\lambda$.
Choose any transversal $\{ t_k \}_{k\in [n]}$  for $[n]$ in $S_n$ with $t_k n = k$ for each $k$.
Then the orbit $\mathcal{W} := \{ \pi_\lambda(t_k) W \}$ is a totally symmetric and real $\operatorname{ECTFF}$.
Furthermore, it is equi-isoclinic if and only if one of the following holds (cf.\ Figure~\ref{fig: single layer Young}):
	\begin{itemize}
	\item[(i)]
	for some integers $a \geq 2$ and $b\geq 1$, $\mu = (b^a)$ is depicted by an $a\times b$ rectangle, and $\lambda = (b+1,b^{a-1})$ adds a box to the first row,
	
	\smallskip
	\item[(ii)]
	for some integers $a \geq 2$ and $b\geq 1$, $\mu = (a^b)$ is 
	depicted by a $b\times a$ rectangle, and $\lambda = (a^b,1)$ adds a box to the bottom,
	
	\smallskip
	\item[(iii)]
	for some integers $a,b \geq 1$ and $c \geq 2$, $\mu = \big( (b+c)^a,b^c \big)$ is 
	depicted by an $(a+c)\times (b+c)$ rectangle with the bottom right $c\times c$ corner removed, and $\lambda = \big( (b+c)^a,b+1,b^{c-1} \big)$ adds a box in the corner.
	\end{itemize}
In that case, $\mathcal{W}$ is a totally symmetric $\operatorname{EITFF}_{\mathbb{R}}(d,r,n)$ with parameters given as follows in cases:
	\begin{itemize}
	\item[(i)--(ii)]
	$\displaystyle
	d = \frac{a}{a+b} rn,
	\qquad
	r =  (ab)! \prod_{k=0}^{a-1} \frac{ k! }{ (k+b)! } = (ab)! \prod_{k=0}^{b-1} \frac{ k! }{ (k+a)! },
	\qquad
	n = {ab+1}$,

	\smallskip
	\item[(iii)]
	\begin{align*}
	d &= \frac{c^2}{(a+c)(b+c)}rn, \\
	r &= (ab+ac+bc)! \prod_{k=0}^{c-1} \frac{ (k!)^2 }{ (a+k)!(b+k)! } \prod_{\ell=0}^{a-1} \frac{ (2c+\ell)! }{ (2c+b+\ell)! } \\
	&= (ab+ac+bc)! \prod_{k=0}^{c-1} \frac{ (k!)^2 }{ (a+k)!(b+k)! } \prod_{\ell=0}^{b-1} \frac{ (2c+\ell)! }{ (2c+a+\ell)! }, \\
	n &= ab+ac+bc+1.
	\end{align*}
	\end{itemize}
\end{theorem}

As we will see, the requirement $W \neq V_\lambda$ precludes taking $a=1$ in (i) and (ii), as well as $c=1$ in (iii).
Table~\ref{tbl: Sn params} lists some small parameters of totally symmetric EITFFs whose existence follows from Theorem~\ref{thm: single layer}.
Fusion synthesis matrices for the first ten columns are included as ancillary files with the arXiv version of this paper.
As in Theorem~\ref{thm: one layer general group}, the subspace ensemble $\mathcal{W}$ in Theorem~\ref{thm: single layer} does not depend on the choice of transversal $\{t_k \}_{k \in [n]}$.
Choosing $t_k = (1\ 2\ \cdots \ n)^k$ shows $\mathcal{W}$ is \textit{harmonic} in the sense of~\cite{FIJM:23}.

\begin{remark}
\label{rem: nice parameters}
For some choices of $a,b,c$, the parameter formulas in Theorem~\ref{thm: single layer} are particularly nice.
If $a=1$ in (i) or (ii), then $(d,r,n) = (b,1,b+1)$, and the EITFF consists of the lines spanned by elements of a regular simplex in $\mathbb{R}^b$.
If $a=2$ in (i) or (ii), then $(d,r,n) = (C_{b+1},C_b,2b+1)$, where $C_b := \tfrac{1}{b+1} \binom{2b}{b}$ is a \textit{Catalan number}.
In particular, $d$ is not divisible by~$r$ in this case unless $b\in\{1,4\}$, since the recurrence relation $C_k = \tfrac{2(2k-1)}{k+1} C_{k-1}$ implies
\[
\tfrac{d}{r} = \tfrac{C_{b+1}}{C_b} = \tfrac{2(2b+1)}{b+2} = 4 - \tfrac{6}{b+2}.
\]
(In fact, this construction accounts for infinitely many choices of coset $\tfrac{d}{r} + \mathbb{Z}$.)
Next, if $b=2$ in (i) or (ii), then it is easier to describe the parameters of the Naimark complementary EITFF, namely, $(nr-d,r,n) = (C_{a+1},C_a,2a+1)$.
Finally, if $a=b=1$ in (iii), then 
$(d,r,n) = \Big( \tfrac{2c^2}{c+1} \binom{2c}{c}, \binom{2c}{c}, 2c+2 \Big)$.
\end{remark}

\begin{table}
\begin{center}
\begin{tabular}{c|rrrrrrrrrrrrr}
$d$ & 5 & 14 & 16 & 42 & 90 & 132 & 168 & 210 & 429 & 448 & 1430 & 2100 & 2112 \\
$r$ & 2 & 5 & 6 & 14 & 20 & 42 & 56 & 42 & 132 & 70 & 429 & 252 & 660 \\
$n$ & 5 & 7 & 6 & 9 & 8 & 11 & 9 & 10 & 13 & 10 & 15 & 12 & 12 \\[3 pt]
$\alpha$ & $\frac{1}{2}$ & $\frac{1}{2}$ & $\frac{1}{2}$ & $\frac{1}{2}$ & $\frac{1}{3}$ & $\frac{1}{2}$ & $\frac{1}{2}$ & $\frac{1}{3}$ & $\frac{1}{2}$ & $\frac{1}{4}$ & $\frac{1}{2}$ & $\frac{1}{5}$ & \rule[-1.8ex]{0pt}{0 pt} $\frac{1}{2}$  \\ \hline
$a$ & \rule{0pt}{2.5ex} 2 & 2 & 1 & 2 & 1 & 2 & 1 & 3 & 2 & 1 & 2 & 1 & 1 \\
$b$ & 2 & 3 & 1 & 4 & 1 & 5 & 2 & 3 & 6 & 1 & 7 & 1 & 3 \\
$c$ & & & 2 & & 3 & & 2 & & & 4 & & 5 & 2
\end{tabular}
\end{center}
\smallskip
\caption{
For each column above, Theorem~\ref{thm: single layer} constructs a totally symmetric $\operatorname{EITFF}_{\mathbb{R}}(d,r,n)$ with isoclinism parameter~$\alpha$.
If a parameter $c$ is given, apply (iii); otherwise, apply (i).
}
\label{tbl: Sn params}
\end{table}

Example~\ref{ex: 5 5 2 redux} illustrates Theorem~\ref{thm: single layer}(i) in the case where $a=b=2$.
Likewise, Example~\ref{ex: 16 6 6} is an instance of Theorem~\ref{thm: single layer}(iii), where $a=b=1$ and $c=2$.
These examples serve as inspiration for the following proof.
(Later, we will generalize other ideas in the proof below as Theorem~\ref{thm: even odd}.)

\begin{proof}[Proof of Theorem~\ref{thm: single layer}]
This construction is an instance of Theorem~\ref{thm: one layer general group}, which implies $\mathcal{W}$ is a totally symmetric ECTFF, and it is equi-isoclinic if and only if the subspaces $W$ and $\pi_\lambda(s_{n-1})W$ are isoclinic.
We test isoclinism with a cross-Gram matrix.
The columns of the $\operatorname{Tab}(\mu) \times \operatorname{Tab}(\lambda)$ matrix $\Psi_{\lambda,\mu}$ form an orthonormal basis for $W$, and the columns of $\pi_\lambda(s_{n-1}) \Psi_{\lambda,\mu}$ form an orthonormal basis for $\pi_\lambda(s_{n-1})W$.
Then $\Psi_{\lambda,\mu}^\top \pi_\lambda(s_{n-1}) \Psi_{\lambda,\mu}$ is a cross-Gram matrix, and for $R,S \in \operatorname{Tab}(\mu)$ the $(S,R)$ entry is
\[
 \langle \Psi_{\lambda,\mu}^\top \pi_\lambda(s_{n-1}) \Psi_{\lambda,\mu} v_R, v_S \rangle
= \langle \pi_\lambda(s_{n-1}) v_{R^\lambda}, v_{S^\lambda} \rangle.
\]
Here, \eqref{eq:Young's orthogonal form} gives
\[
\pi_\lambda(s_{n-1}) v_{R^\lambda}
= \frac{1}{D_{R^\lambda}(n,n-1)} v_{R^\lambda} + \sqrt{ 1 - \frac{1}{D_{R^\lambda}(n,n-1)^2} } v_{s_{n-1} R^\lambda }.
\]
To compute its inner product with $v_{S^\lambda}$, observe that $\delta_{R^\lambda,S^\lambda} = \delta_{R,S}$, and $s_{n-1} R^\lambda  \neq S^\lambda$ since these tableaux differ in the $\lambda-\mu$ entry.
Therefore,
\begin{equation}
\label{eq: cross-Gram}
\big \langle \Psi_{\lambda,\mu}^\top \pi_\lambda(s_{n-1}) \Psi_{\lambda,\mu} v_R, v_S \big \rangle
= \frac{1}{D_{R^\lambda}(n,n-1)} \delta_{R,S},
\qquad
R,S \in \operatorname{Tab}(\mu).
\end{equation}
In particular, the cross-Gram matrix is diagonal.
Its product with its adjoint is a multiple of the identity
if and only if the diagonal has constant absolute value, if and only if $| D_{R^\lambda}(n,n-1)|$ takes the same value for all $R \in \operatorname{Tab}(\mu)$.
To interpret this condition in terms of the underlying Young diagrams, observe that $n$ appears in box $(k,l):=\lambda - \mu$ of $R^\lambda$, while $n-1$ necessarily appears in a removable box of $\mu$.
Furthermore, every removable box of $\mu$ occurs in this way for some $R \in \operatorname{Tab}(\mu)$ by~\eqref{eq: removable boxes get n}.
Overall, $\mathcal{W}$ is equi-isoclinic if and only if 
there is an integer $m \geq 1$ such that every removable box of $\mu$ has axial distance $m$ or $-m$ to $(k,l)$.

As we now explain, this condition is satisfied for each of (i)--(iii).
In (i) and (ii), $\mu$ is a rectangle with only one removable box, and the condition is trivially satisfied.
In (iii), $\mu$ has exactly two removable boxes.
The first is at position $(a,b+c)$, with axial distance $-c$ from $(k,l) = (a+1,b+1)$.
The second is at position $(a+c,b)$, with axial distance $c$ from $(k,l)$.
Hence, each of (i)--(iii) yields a totally symmetric $\operatorname{EITFF}(d,n,r)$, where the given values of $d=d_\lambda$ and $r=d_\mu$ can be verified by straightforward (if tedious) application of the Hook length formula.

Conversely, suppose $\lambda$ and $\mu$ are chosen so that $\mathcal{W}$ is equi-isoclinic, and let $m$ be the common absolute value of the axial distance to any removable box of $\mu$ from $(k,l) = \lambda - \mu$.
Then every removable box of $\mu$ appears on one of two superdiagonals: $l-k-m$ or $l-k+m$.
Any superdiagonal contains at most one removable box of $\mu$, so $\mu$ has at most two removable boxes.
We now argue in cases.
First suppose $\mu$ has only one removable box.
Then $\mu$ is a constant sequence, and its Young diagram is a rectangle.
There are only two places to add a box and produce the Young diagram of $\lambda \in \mu^\uparrow$.
Depending on the choice, $(\lambda,\mu)$ is as in (i) or (ii), where the condition $a \geq 2$ follows from the constraint $W \neq V_\lambda$.
Now suppose $\mu$ has exactly two removable boxes.
Then $\mu$ is a sequence with exactly two values, and we can write $\mu = (a^{b+d},b^c)$ for some $a,b,c,d\geq 1$.
Its Young diagram has removable boxes at $(a,b+d)$ and $(a+c,b)$, as shown in blue below (with some extra red boxes).
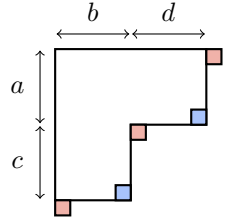
\begin{center}
\begin{tikzpicture}

\draw[thick] (0,0) -- (1,0) -- (1,1) -- (2,1) -- (2,2) -- (0,2) -- cycle;

\filldraw[thick,fill=ltblue] (1,0) rectangle (1-0.2,0+0.2);
\filldraw[thick,fill=ltblue] (2,1) rectangle (2-0.2,1+0.2);
\filldraw[thick,fill=ltpink] (1,1) rectangle (1+0.2,1-0.2);
\filldraw[thick,fill=ltpink] (0,0) rectangle (0+0.2,0-0.2);
\filldraw[thick,fill=ltpink] (2,2) rectangle (2+0.2,2-0.2);

\draw[<->] (-0.2,0+0.02) -- (-0.2,1-0.02);
\draw[<->] (-0.2,1+0.02) -- (-0.2,2-0.02);

\node[] at (-0.5,0.5) {$c$};
\node[] at (-0.5,1.5) {$a$};

\draw[<->] (0+0.02,2+0.2) -- (1-0.02,2+0.2);
\draw[<->] (1+0.02,2+0.2) -- (2-0.02,2+0.2);

\node[] at (0.5,2.5) {$b$};
\node[] at (1.5,2.5) {$d$};

\end{tikzpicture} 
\end{center}
There are three possibilities for $(k,l) = \lambda - \mu$, shown in red above.
If $(k,l) = (1,b+d+1)$ were at the top, then the unequal axial distances would both be positive.
If $(k,l) = (a+c+1,1)$ were at the bottom, then the unequal axial distances would both be negative.
Neither can happen, so $(k,l) = (a+1,b+1)$ is in the corner.
Then the axial distances to the removable boxes are $-d$ and $c$.
These must be opposites, so $d=c$.
Finally, the constraint $W \neq V_\lambda$ ensures $c \geq 2$: when $c=1$, $\lambda$ has a rectangular Young diagram with just one removable box, and the branching rule implies $\pi_\lambda \big|_{S_{n-1}} \cong \pi_\mu$, so that $\dim V_\lambda = \dim V_\mu = \dim W$.
Overall, $(\lambda,\mu)$ is as in (iii).
This completes the proof.
\end{proof}

\section{Total symmetry: Multiple layers}
\label{sec: total symmetry multiple layers}

Next we consider totally symmetric subspace ensembles that arise from the multiple-layer construction of Theorem~\ref{thm: multiple layers general group}, obtaining additional infinite families of totally symmetric EITFFs in Theorems~\ref{thm: three parts} and~\ref{thm: four parts}.

To begin, we interpret Theorem~\ref{thm: multiple layers general group} for representations of the symmetric group, and find conditions under which the resulting subspace ensemble is equi-isoclinic.

\begin{theorem}
\label{thm: multiple layers}
Select a partition $\mu \vdash (n-1)$ and a subset $L \subseteq \mu^\uparrow$, and consider the representation $\pi_L := \bigoplus_{\lambda \in L} \pi_{\lambda}$ of $S_n$ on $V_L := \bigoplus_{\lambda \in L} V_{\lambda}$ with degree $d_L:=\sum_{\lambda \in L} d_{\lambda}$.
Next, let $W_L \leq V_L$ be the image of the isometry $\Psi_L \colon V_\mu \to V_L$ given by 
\[ \Psi_L(v) = \Big\{ \sqrt{d_{\lambda} / d_L } \,  \Psi_{\lambda,\mu} (v) \Big\}_{\lambda \in L}, \qquad v \in V_\mu. \]
Finally, choose any transversal $\{ t_k \}_{k\in [n]}$ for $[n]$ in $S_n$ with $t_k n = k$ for each $k$.
Then the orbit $\mathcal{W}_L := \{ \pi_L(t_k) W_L \}_{k\in [n]}$ is a totally symmetric $\operatorname{ECTFF}_{\mathbb{R}}(d_L,d_\mu,n)$.
Furthermore, $\mathcal{W}_L$ is equi-isoclinic if and only if $L$ satisfies
\begin{equation}
\label{eq: distance condition}
\exists \, \beta \geq 0 \text{ s.t.\ } 
\sum_{\lambda \in L} \frac{d_\lambda}{n d_\mu} \cdot \frac{ 1 }{ D\big( \lambda - \mu, (k,l) \big) } 
= 
\pm 
\beta
\qquad
\forall (k,l) \in \lfloor \mu \rfloor,
\end{equation}
in which case $\frac{ n d_\mu}{ d_L} \beta$ is the common isoclinism parameter of $\mathcal{W}_L$ and
\[
\beta = \sqrt{ \frac{d_L( n d_\mu - d_L)}{d_\mu^2 n^2 (n-1) } }.
\]
\end{theorem}

We discuss how Theorem~\ref{thm: multiple layers} relates to Theorem~\ref{thm: single layer} in Example~\ref{ex: single layer redux} below.
For simplicity, we do not impose $W_L \neq V_L$ in Theorem~\ref{thm: multiple layers}.
While the statement of Theorem~\ref{thm: multiple layers} could be modified to remove the global factor of $\tfrac{1}{n d_\mu}$ from the left-hand side of~\eqref{eq: distance condition}, this factor makes it easier to verify~\eqref{eq: distance condition} in practice.
Indeed, the hook length formula implies $\frac{ d_\lambda }{ n d_\mu }$ is the product of all hook lengths for $\mu$ divided by the product of all hook lengths for $\lambda$.
Many of these hook lengths are the same for $\mu$ and $\lambda \in \mu^\uparrow$ (and so they cancel out), and many others telescope away in the quotient. 
A version of Theorem~\ref{thm: multiple layers} was announced by the authors as Theorem~2 of~\cite{FIJM:ICASSP}, together with a promise to prove it in a later manuscript.
The formulation given above is more detailed, and its proof below settles our debt from~\cite{FIJM:ICASSP}.

\begin{proof}[Proof of Theorem~\ref{thm: multiple layers}]
This construction is an instance of Theorem~\ref{thm: multiple layers general group}, which implies $\mathcal{W}_L$ is a totally symmetric $\operatorname{ECTFF}(d,d_\mu,n)$, and it is an EITFF if and only if the subspaces $W_L$ and $\pi_L(s_{n-1}) W_L$ are isoclinic.
To determine isoclinism, we consider a cross-Gram matrix.
The operator $\Psi_L$ is an isometry onto $W_L$, and $\pi_L(s_{n-1}) \Psi_L$ is an isometry onto $\pi_L(s_{n-1}) W_L$.
Then $\Psi_L^\top \pi_L(s_{n-1}) \Psi_L \colon V_\mu \to V_\mu$ is a cross-Gram operator, where
\[ 
\Psi_L^\top \pi_L(s_{n-1}) \Psi_L = \sum_{\lambda \in L} \frac{d_{\lambda}}{d_L} \Psi_{\lambda,\mu}^\top \pi_{\lambda}(s_{n-1}) \Psi_{\lambda,\mu}. 
\]
By~\eqref{eq: cross-Gram}, the corresponding cross-Gram matrix has entries
\[
\big \langle \Psi_L^\top \pi_L(s_{n-1}) \Psi_L v_R, v_S \big \rangle
=
\frac{1}{d_L}
\sum_{\lambda \in L} \frac{d_{\lambda}}{D_{R^\lambda}(n,n-1) } \delta_{R,S},
\qquad
R,S \in \operatorname{Tab}(\mu).
\]
In particular, the cross-Gram matrix is diagonal.
Its product with its adjoint is a multiple of the identity if and only if its diagonal has constant absolute value, if and only if there exists $\beta \geq 0$ such that $\Big| \sum_{\lambda \in L} \frac{ d_\lambda }{ n d_\mu }\cdot \frac{1}{D_{R^\lambda}(n,n-1) } \Big| = \beta$ for every $R \in \operatorname{Tab}(\mu)$.
To see this is equivalent to~\eqref{eq: distance condition}, fix $R \in \operatorname{Tab}(\mu)$ and notice that $n$ always appears in the $\lambda - \mu$ box of $R^\lambda$, while $n-1$ appears in a removable box $(k,l) \in \lfloor \mu \rfloor$ whose location depends only on $R$ and not on $\lambda$.
Then $D_{R^\lambda}(n,n-1) = D\big(\lambda-\mu, (k,l) \big)$ for every $\lambda \in L$.
Furthermore, every removable box $(k,l) \in \lfloor \mu \rfloor$ occurs in this way for some $R \in \operatorname{Tab}(\mu)$ by~\eqref{eq: removable boxes get n}.
Overall, $\mathcal{W}_L$ is equi-isoclinic if and only if~\eqref{eq: distance condition} holds, where $\tfrac{ n d_\mu }{ d_L } \beta$ is the parameter of isoclinism for $W_L$ and $\pi_L(s_{n-1})W_L = \pi_L(t_{n-1}) W_L$.
Then the formula for $\beta$ follows from~\eqref{eq: isoclinism parameter}.
\end{proof}

\begin{example}
\label{ex: single layer redux}
Suppose $L = \{ \lambda \}$ is a singleton in Theorem~\ref{thm: multiple layers}.
Then $\mathcal{W}_L$ is precisely the ECTFF of Theorem~\ref{thm: single layer}.
Furthermore,~\eqref{eq: distance condition} is easily seen to be satisfied in each of Theorem~\ref{thm: single layer}(i)--(iii).
In (i) and (ii), $\mu$ has only one removable box, and so the quantity on the left-hand side of~\eqref{eq: distance condition} is trivially constant on $\lfloor \mu \rfloor$.
In (iii), the two removable boxes of $\mu$ are at opposite axial distances from $\lambda-\mu$, and so the respective quantities from the left-hand side of~\eqref{eq: distance condition} are opposites.
\end{example}

\begin{example}
\label{ex: orthogonal subspaces}
Let $\mu \vdash (n-1)$ be arbitrary, and take $L$ to be all of $\mu^\uparrow$ in Theorem~\ref{thm: multiple layers}.
Then $d_{\mu^\uparrow} = \sum_{\lambda \in \mu^\uparrow} d_\lambda = n d_\mu$, by the remark following Theorem~4 in~\cite{V:06}.
Consequently, the subspace ensemble $\mathcal{W}_{\mu^\uparrow}$ is an $\operatorname{ECTFF}(nd_\mu,d_\mu,n)$.
It follows that $\mathcal{W}_{\mu^\uparrow}$ consists of $n$ mutually orthogonal subspaces: since equality holds in~\eqref{eq: chordal welch}, the chordal distance between any two subspaces of $\mathcal{W}_{\mu^\uparrow}$ is exactly $r=\sqrt{d_\mu}$, and all the principal angles are right angles.
In particular, $\mathcal{W}_{\mu^\uparrow}$ is an EITFF.
\end{example}

As we will see in Theorem~\ref{thm: even odd} below, for any $\mu \vdash (n-1)$ there are either zero or exactly two choices of nonempty proper subsets $L \subset \mu^\uparrow$ that produce EITFFs in Theorem~\ref{thm: multiple layers}, and in the latter case the two choices are set complements.
The following explains the relation of the corresponding EITFFs.

\begin{lemma}
Given $\mu \vdash (n-1)$ and $L \subseteq \mu^\uparrow$, denote $L^c:=\mu^\uparrow \setminus L$.
Then the ECTFFs $\mathcal{W}_L$ and $\mathcal{W}_{L^c}$ produced by Theorem~\ref{thm: multiple layers} are Naimark complements.
In particular, $\mathcal{W}_L$ is equi-isoclinic if and only if $\mathcal{W}_{L^c}$ is equi-isoclinic.
\end{lemma}

\begin{proof}
Fix a transversal $\{ t_k \}_{k\in [n]}$ for $[n]$ in $S_n$, where $t_k n = k$ for each $k \in [n]$.
We continue in the notation of Theorem~\ref{thm: multiple layers}.
For each partition $\lambda$, we identify $V_\lambda \cong \mathbb{R}^{\operatorname{Tab}(\lambda)}$ and consider $\pi_\lambda$ to take values in the space of real $\operatorname{Tab}(\lambda) \times \operatorname{Tab}(\lambda)$ matrices.
Then $V_L \cong \mathbb{R}^{\cup_{\lambda \in L} \operatorname{Tab}(\lambda)}$, and $\Psi_L$ may be viewed as a $( \cup_{\lambda \in L} \operatorname{Tab}(\lambda) ) \times \operatorname{Tab}(\mu)$ matrix.
Analogously, by considering the disjoint union $\mu^\uparrow = L \cup L^c$, we may express
\[
\Psi_{\mu^\uparrow} = 
\left[ \begin{array}{c} 
\sqrt{ \tfrac{ d_L }{ d_{\mu^\uparrow} } } \Psi_L \\[8 pt]
\sqrt{ \tfrac{ d_{L^c} }{ d_{\mu^\uparrow} } } \Psi_{L^c} 
\end{array} \right].
\]
Then since $\pi_{\mu^\uparrow} = \pi_L \oplus \pi_{L^c}$, we have
\begin{align*}
\Phi_{\mu^\uparrow} &:=
\left[ \begin{array}{ccc}
\pi_{\mu^\uparrow}(t_1) \Psi_{\mu^\uparrow} & \cdots & \pi_{\mu^\uparrow}(t_n) \Psi_{\mu^\uparrow}
\end{array} \right] \\
&=
\left[ \begin{array}{ccc}
\sqrt{ \tfrac{ d_L }{ d_{\mu^\uparrow} } } \pi_L(t_1) \Psi_L & \cdots & \sqrt{ \tfrac{ d_L }{ d_{\mu^\uparrow} } } \pi_L(t_n) \Psi_L \\[10 pt]
\sqrt{ \tfrac{ d_{L^c} }{ d_{\mu^\uparrow} } } \pi_{L^c}(t_1) \Psi_{L^c} & \cdots & \sqrt{ \tfrac{ d_{L^c} }{ d_{\mu^\uparrow} } } \pi_{L^c}(t_n) \Psi_{L^c} 
\end{array} \right]
=: \left[ \begin{array}{ccc}
\sqrt{ \tfrac{ d_L }{ d_{\mu^\uparrow} } } \Phi_L \\[10 pt]
\sqrt{ \tfrac{ d_{L^c} }{ d_{\mu^\uparrow} } } \Phi_{L^c}
\end{array} \right],
\end{align*}
where $\Phi_{\mu^\uparrow}$, $\Phi_L$, and $\Phi_{L^c}$ are fusion synthesis matrices for $\mathcal{W}_{\mu^\uparrow}$, $\mathcal{W}_L$, and $\mathcal{W}_{L^c}$, respectively.
By Example~\ref{ex: orthogonal subspaces}, $\Phi_{\mu^\uparrow}$ is an orthogonal matrix and $d_{\mu^\uparrow} = n d_\mu$.
Thus,
\[
I = \Phi_{\mu^\uparrow}^\top \Phi_{\mu^\uparrow} 
= \tfrac{ d_L }{ n d_\mu } \Phi_L^\top \Phi_L + \tfrac{ d_{L^c} }{ n d_\mu } \Phi_{L^c}^\top \Phi_{L^c},
\]
so $\mathcal{W}_L$ and $\mathcal{W}_{L^c}$ are Naimark complements.
\end{proof}

\begin{example}
In Theorem~\ref{thm: single layer}, the EITFFs of types~(i) and~(ii) are Naimark complements after reversing the parameters $a$ and $b$.
For type~(iii), we have $|\mu^\uparrow| = 3$ and so a Naimark complement $\mathcal{W}_L$ is constructed using two layers in Theorem~\ref{thm: multiple layers}, namely, $L = \{ \mu + (1,b+c+1), \mu+(a+c+1,1) \}$.
\end{example}

\subsection{Isoclinic partitions}

\begin{definition}
We say a partition $\mu \vdash (n-1)$ is \textbf{isoclinic} if there exists a proper and nonempty subset $L \subset \mu^\uparrow$ for which the subspace ensemble $\mathcal{W}_L$ of Theorem~\ref{thm: multiple layers} is equi-isoclinic, hence an EITFF.
\end{definition}

\begin{lemma}
If $\mu \vdash (n-1)$ is an isoclinic partition, then so is its transpose $\mu'$.
\end{lemma}

\begin{proof}
We apply Theorem~\ref{thm: multiple layers}.
Since $\mu$ is isoclinic, there is a proper and nonempty subset $L \subset \mu^\uparrow$ for which~\eqref{eq: distance condition} holds.
Consider the proper and nonempty subset $L' := \{ \lambda' : \lambda \in L\} \subset (\mu')^\uparrow$.
For each $\lambda \in L$, if $\lambda - \mu =: (i,j)$ then $\lambda' - \mu' = (j,i)$ and $D\big( \lambda - \mu, (k,l) \big) = - D\big( \lambda' - \mu', (l,k) \big)$ for each $(k,l) \in \lfloor \mu \rfloor$.
Since $d_{\lambda'} = d_\lambda$ and $d_{\mu'} = d_\mu$, it follows that
\[
\sum_{\lambda \in L} \frac{d_\lambda}{n d_\mu} \cdot \frac{ 1 }{ D\big( \lambda - \mu, (k,l) \big) }
= -  \sum_{\lambda \in L} \frac{d_{\lambda'}}{n d_{\mu'}} \cdot \frac{ 1 }{ D\big( \lambda' - \mu', (l,k) \big) },
\qquad
(k,l) \in \lfloor \mu \rfloor.
\]
Then since $L$ satisfies~\eqref{eq: distance condition} and $\lfloor \mu' \rfloor = \{ (l,k) : (k,l) \in \lfloor \mu \rfloor \}$, the subspace ensemble $\mathcal{W}_{L'}$ corresponding to $L'$ in Theorem~\ref{thm: multiple layers} is equi-isoclinic.
\end{proof}

For any $\mu \vdash (n-1)$, every subset of $\mu^\uparrow$ delivers an ECTFF via Theorem~\ref{thm: multiple layers}.
However, the following theorem shows that at most two proper nonempty subsets of $\mu^\uparrow$ can deliver EITFFs.
Furthermore, there is no mystery about the identities of these sets, or even about the signs of the sums in~\eqref{eq: distance condition}.

\begin{figure}

\begin{center}
\begin{tikzpicture}

\draw[thick] (0,0) -- (1,0) -- (1,1) -- (2,1) -- (2,2) -- (3,2) -- (3,3) -- (4,3) -- (4,4) -- (0,4) -- cycle;

\filldraw[thick,fill=ltblue] (1,0) rectangle (1-0.2,0+0.2);
\filldraw[thick,fill=ltblue] (2,1) rectangle (2-0.2,1+0.2);
\filldraw[thick,fill=ltblue] (3,2) rectangle (3-0.2,2+0.2);
\filldraw[thick,fill=ltblue] (4,3) rectangle (4-0.2,3+0.2);

\filldraw[thick,fill=ltpink] (0,0) rectangle (0+0.2,0-0.2);
\filldraw[thick,fill=ltpink] (2,2) rectangle (2+0.2,2-0.2);
\filldraw[thick,fill=ltpink] (3,3) rectangle (3+0.2,3-0.2);
\filldraw[thick,fill=ltpink] (4,4) rectangle (4+0.2,4-0.2);

\node[right,blue] at (1+0.0,0+0.05) {\small{$(k_m,l_m)$}};
\node[right,blue] at (2+0.0,1+0.05) {\small{$(k_3,l_3)$}};
\node[right,blue] at (3+0.0,2+0.05) {\small{$(k_2,l_2)$}};
\node[right,blue] at (4+0.0,3+0.05) {\small{$(k_1,l_1)$}};

\node[right,red] at (0+0.1,0-0.35) {\small{$\lambda_{m+1}-\mu$}};
\node[right,red] at (2+0.1,2-0.35) {\small{$\lambda_3-\mu$}};
\node[right,red] at (3+0.1,3-0.35) {\small{$\lambda_2-\mu$}};
\node[right,red] at (4+0.2,4-0.15) {\small{$\lambda_1-\mu$}};

\node[] at (1.85+0,0.5+0) {$\cdot$};
\node[] at (1.85+0.07,0.5+0.07) {$\cdot$};
\node[] at (1.85+2*0.07,0.5+2*0.07) {$\cdot$};

\end{tikzpicture} 
\end{center}

\caption{Let $\mu$ be a partition with $m$ distinct parts.
Then we label $\lfloor \mu \rfloor = \{ (k_1,l_1),\dotsc,(k_m,l_m) \}$ by descending superdiagonal order, and $\mu^\uparrow = \{ \lambda_1,\dotsc,\lambda_{m+1}\}$ by descending superdiagonal order of differences with $\mu$.}

\label{fig: label order}

\end{figure}
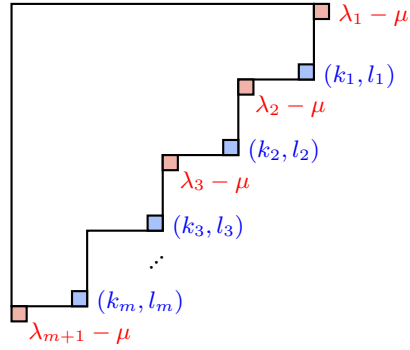

\begin{theorem}
\label{thm: even odd}
Suppose $\mu \vdash (n-1)$ is an isoclinic partition. 
Enumerate $\lfloor \mu \rfloor = \{ (k_1,l_1),\dotsc,(k_m,l_m)\}$ and 
$\mu^\uparrow = \{ \lambda_1,\dotsc,\lambda_{m+1} \}$ as in Figure~\ref{fig: label order}.
Then the following hold:
	\begin{itemize}
	\item[(a)]
	there are exactly two choices of nonempty proper subset $L \subset \mu^\uparrow$ for which the subspace ensemble $\mathcal{W}_L$ of Theorem~\ref{thm: multiple layers} is equi-isoclinic, namely, 
	\[
	L_0:=\{ \lambda_p : p \in [m+1] \text{ is even}\}
	\qquad
	\text{and}
	\qquad
	L_1:=\{ \lambda_p : p \in [m+1] \text{ is odd}\},
	\]
	
	\medskip
	
	\item[(b)]
	for $\delta \in \{0,1\}$, let $\frac{ n d_\mu }{ d_{L_\delta} } \beta$ be the common isoclinism parameter of $\mathcal{W}_{L_\delta}$, where $\beta \geq 0$;
	then
	\[
	\sum_{\lambda \in L_\delta} \frac{ d_\lambda }{ n d_\mu } \cdot \frac{ 1 }{ D \big( \lambda - \mu, (k_q,l_q) \big) } = (-1)^{q+\delta} \beta,
	\qquad
	q \in [m].
	\]
	\end{itemize}

\end{theorem}

\begin{proof}
For (a), let $L \subset \mu^\uparrow$ be a nonempty and proper subset of $\mu^\uparrow$ such that $\mathcal{W}_L$ is equi-isoclinic, i.e.,~\eqref{eq: distance condition} holds.
We will prove that the complement $L^c:=\mu^\uparrow \setminus L$ contains no adjacent terms from the sequence $\lambda_1,\dotsc,\lambda_{m+1}$.
Since the Naimark complement $\mathcal{W}_{L^c}$ is also equi-isoclinic, it will follow that $L$ contains no adjacent terms from this sequence either; thus, the sequence $\lambda_1,\dotsc,\lambda_{m+1}$ alternates between elements of $L$ and $L^c$, and $\{ L, L^c \} = \{ L_0, L_1 \}$.

To begin, we fix $p \in [m+1]$ and prove the crucial inequality
\begin{equation}
\label{eq: even odd 1}
\frac{ d_{\lambda_p} }{ D\big(\lambda_p - \mu, (k_{q-1},l_{q-1}) \big) } > \frac{ d_{\lambda_p} }{ D\big(\lambda_p - \mu, (k_{q},l_{q}) \big) },
\qquad
1 < q \neq p.
\end{equation}
To see this, 
fix $1 < q \neq p$ and observe that
\[
{D\big(\lambda_p - \mu, (k_{q-1},l_{q-1}) \big)} < {D\big(\lambda_p - \mu, (k_{q},l_{q}) \big)},
\qquad
q > 1.
\]
If $p > q$, then ${D\big(\lambda_p - \mu, (k_{q-1},l_{q-1}) \big)} < {D\big(\lambda_p - \mu, (k_{q},l_{q}) \big)} < 0$, and~\eqref{eq: even odd 1} follows.
Likewise, if $p \leq q-1$, then $0 < {D\big(\lambda_p - \mu, (k_{q-1},l_{q-1}) \big)} < {D\big(\lambda_p - \mu, (k_{q},l_{q}) \big)}$, and~\eqref{eq: even odd 1} follows again.
This proves the claim.

Now suppose $L^c$ contains one of the middle terms $\lambda_{q}$ with $1 < q < m+1$.
Then any $\lambda_p \in L$ has $p \neq q$, and summing over all such $p$ in~\eqref{eq: even odd 1} shows that
\begin{equation}
\label{eq: even odd 2}
\sum_{\lambda \in L} \frac{ d_{\lambda} }{ D\big(\lambda - \mu, (k_{q-1},l_{q-1}) \big) }
>
\sum_{\lambda \in L} \frac{ d_{\lambda} }{ D\big(\lambda - \mu, (k_{q},l_{q}) \big) },
\qquad
1 < q < m+1, \ \lambda_{q} \in L^c.
\end{equation}
If $L^c$ were to contain two adjacent middle terms $\lambda_q, \lambda_{q+1}$ with $1 < q < m$, it would follow that
\[
\sum_{\lambda \in L} \frac{ d_{\lambda} }{ D\big(\lambda - \mu, (k_{q-1},l_{q-1}) \big) }
>
\sum_{\lambda \in L} \frac{ d_{\lambda} }{ D\big(\lambda - \mu, (k_{q},l_{q}) \big) }
>
\sum_{\lambda \in L} \frac{ d_{\lambda} }{ D\big(\lambda - \mu, (k_{q+1},l_{q+1}) \big) }.
\]
This cannot occur since three distinct real numbers cannot share one absolute value in~\eqref{eq: distance condition}.
It remains to show that $L^c$ contains neither $\{ \lambda_1, \lambda_2 \}$ nor $\{ \lambda_m, \lambda_{m+1} \}$.

For the first case, suppose $\lambda_1,\dotsc,\lambda_q \in L^c$.
Then any $\lambda_p \in L$ has $p > q$, so that $D\big( \lambda_p - \mu, (k_q,l_q) \big) < 0$.
Summing over all $\lambda_p \in L$, we obtain
\begin{equation}
\label{eq: even odd 3}
\sum_{\lambda \in L} \frac{ d_{\lambda} }{ D\big(\lambda - \mu, (k_{q},l_{q}) \big) } < 0,
\qquad
\lambda_1,\dotsc,\lambda_q \in L^c.
\end{equation}
If $L^c$ were to contain both $\lambda_1$ and $\lambda_2$, then we could take 
$q=2$
in~\eqref{eq: even odd 1} and sum over all $\lambda_p \in L$ to conclude
\[
0 
> \sum_{\lambda \in L} \frac{ d_{\lambda} }{ D\big(\lambda - \mu, (k_1,l_1) \big) } 
> \sum_{\lambda \in L} \frac{ d_{\lambda} }{ D\big(\lambda - \mu, (k_2,l_2) \big) }.
\]
This is impossible since distinct negative numbers cannot share an absolute value in~\eqref{eq: distance condition}.

The story for the second case is similar.
If $\lambda_q,\dotsc,\lambda_{m+1} \in L^c$, then any $\lambda_p \in L$ has $p < q$, so that $D\big( \lambda - \mu, (k_{q-1},l_{q-1}) \big) > 0$.
Hence,
\begin{equation}
\label{eq: even odd 4}
\sum_{\lambda \in L} \frac{ d_{\lambda} }{ D\big(\lambda - \mu, (k_{q-1},l_{q-1}) \big) } > 0,
\qquad
\lambda_q,\dotsc,\lambda_{m+1} \in L^c.
\end{equation}
If $L^c$ contained both $\lambda_m$ and $\lambda_{m+1}$, and if $m$ was greater than $1$ (so that  we were not in the first case), then we could take $q=m$ in~\eqref{eq: even odd 1} and sum over all $\lambda_p \in L$ to conclude
\[
\sum_{\lambda \in L} \frac{ d_{\lambda} }{ D\big(\lambda - \mu, (k_{m-1},l_{m-1}) \big) } 
> \sum_{\lambda \in L} \frac{ d_{\lambda} }{ D\big(\lambda - \mu, (k_m,l_m) \big) }
> 0,
\]
which contradicts~\eqref{eq: distance condition}.
This completes the proof of (a).

For (b), we prove the case where $\delta = 0$; a similar argument applies when $\delta = 1$.
Take $L = L_0$ above, so that $L^c = L_1$.
It suffices to establish the sign of each sum
\[
c_q:=\sum_{\lambda \in L} \frac{ d_\lambda }{ n d_\mu } \cdot \frac{ 1 }{ D\big(\lambda - \mu, (k_{q},l_{q}) \big) },
\qquad q \in [m],
\]
since each $c_q = \pm \beta$ by~\eqref{eq: distance condition} and (a).
We have
$c_1 < 0$ by~\eqref{eq: even odd 3} since $\lambda_1 \in L^c$.
Next, if $3 \leq q \leq m$ is odd, then~\eqref{eq: even odd 2} implies $c_{q-1} > 0$ and $c_q < 0$.
So far we have established the sign of $c_q$ for every $q \in [m]$ except for $c_m$ when $m$ is even.
In that remaining case, we have $\lambda_{m+1} \in L^c$, and~\eqref{eq: even odd 4} yields $c_m > 0$.
\end{proof}

The remainder of this section is devoted to identifying isoclinic partitions with the following characterization, which follows immediately from Theorems~\ref{thm: multiple layers} and~\ref{thm: even odd}.

\begin{corollary}
\label{cor: isoclinic partition}
Given $\mu \vdash (n-1)$, enumerate $\lfloor \mu \rfloor = \{ (k_1,l_1),\dotsc,(k_m,l_m)\}$ and 
$\mu^\uparrow = \{ \lambda_1,\dotsc,\lambda_{m+1} \}$ as in Figure~\ref{fig: label order}.
Define $L_0:=\{ \lambda_p : p \in [m+1] \text{ is even}\}$ and $L_1:=\{ \lambda_p : p \in [m+1] \text{ is odd} \}$.
Then for either choice of $\delta \in \{0,1\}$, $\mu$ is an isoclinic partition if and only if
\begin{equation}
\label{eq: isoclinic partition}
\exists \, \beta \geq 0 \text{ s.t.\ } 
\sum_{\lambda \in L_\delta} \frac{ d_\lambda }{ n d_\mu } \cdot \frac{ 1 }{ D\big( \lambda - \mu, (k_q,l_q) \big) } = (-1)^{q+\delta} \beta
\qquad
\forall q \in [m].
\end{equation}
In that case, the subspace ensemble $\mathcal{W}_{L_\delta}$ from Theorem~\ref{thm: multiple layers} is a totally symmetric $\operatorname{EITFF}_{\mathbb{R}}(d_{L_\delta},d_\mu,n)$, and $\frac{ n d_\mu }{ d_{L_\delta} } \beta$ is its common parameter of isoclinism, where
\[
\beta = \sqrt{ \frac{d_{L_\delta}( n d_\mu-d_{L_\delta})}{d_\mu^2 n^2 (n-1) } }.
\]
\end{corollary}

\subsection{Examples}
We organize our search for isoclinic partitions in terms of the number $m$ of distinct parts.
When an isoclinic partition $\mu$ has $m=1$ or $m=2$ distinct parts, Theorem~\ref{thm: even odd} implies one of the two resulting nontrivial EITFFs must be constructed using a single layer, and so $\mu$ is described in Theorem~\ref{thm: single layer}.
For partitions with $m=3$ distinct parts, we have the following characterization.

\begin{theorem}
\label{thm: three parts}
For $a,b,c,e,f,g \in \mathbb{N}$, the partition $\mu := \big({(e+f+g)}^a,{(e+f)}^b,e^c\big)$ shown below (with some extra red boxes) is isoclinic if and only if 
\begin{equation}
\label{eq: three parts}
c = \frac{2af}{b+g} + f
\qquad
\text{and}
\qquad
e = \frac{(a+g)f}{b} - c.
\end{equation}
	
\begin{center}
\begin{tikzpicture}

\draw[thick] (0,0) -- (1,0) -- (1,1) -- (2,1) -- (2,2) -- (3,2) -- (3,3) -- (0,3) -- cycle;

\filldraw[thick,fill=ltblue] (1,0) rectangle (1-0.2,0+0.2);
\filldraw[thick,fill=ltblue] (2,1) rectangle (2-0.2,1+0.2);
\filldraw[thick,fill=ltblue] (3,2) rectangle (3-0.2,2+0.2);
\filldraw[thick,fill=ltpink] (1,1) rectangle (1+0.2,1-0.2);
\filldraw[thick,fill=ltpink] (3,3) rectangle (3+0.2,3-0.2);

\draw[<->] (-0.2,0+0.02) -- (-0.2,1-0.02);
\draw[<->] (-0.2,1+0.02) -- (-0.2,2-0.02);
\draw[<->] (-0.2,2+0.02) -- (-0.2,3-0.02);

\node[] at (-0.5,0.5) {$c$};
\node[] at (-0.5,1.5) {$b$};
\node[] at (-0.5,2.5) {$a$};

\draw[<->] (0+0.02,3+0.2) -- (1-0.02,3+0.2);
\draw[<->] (1+0.02,3+0.2) -- (2-0.02,3+0.2);
\draw[<->] (2+0.02,3+0.2) -- (3-0.02,3+0.2);

\node[] at (0.5,3.5) {$e$};
\node[] at (1.5,3.5) {$f$};
\node[] at (2.5,3.5) {$g$};

\end{tikzpicture} 
\end{center}
\end{theorem}

\begin{remark}
Theorem~\ref{thm: three parts} produces at least one isoclinic partition (hence at least one totally symmetric real EITFF) for every choice of 
$a,f \in \mathbb{N}$ with $af > 1$.
Indeed, the following procedure yields positive integers 
$a,b,c,e,f,g \in \mathbb{N}$ that satisfy~\eqref{eq: three parts}:
	\begin{itemize}
	\item[\textsc{Step 1:}]
	Choose any $a,f \in \mathbb{N}$ with $af > 1$.
	\item[\textsc{Step 2:}]
	Choose any divisor $h \mid 2af$ with $h > 2$.
	\item[\textsc{Step 3:}]
	Choose any divisor $b \mid (a+h)f$ with $0 < b < \tfrac{h}{2}$.
	\item[\textsc{Step 4:}]
	Put $c := f + \tfrac{2af}{h}$, $e := \tfrac{ (a-b+h)f }{b} - c$, and $g := h-b$.
	\end{itemize}
In Step~2, such $h$ exists since $2af > 2$.
In Step~3, such $b$ exists since we could take ${b = 1}$.
To see that $e$ is a positive integer in Step~4, observe that $b < \tfrac{h}{2}$ implies $\tfrac{(a+h)f}{b} > \tfrac{2(a+h)f}{h} = f + c$, and so $c$ is less than the integer ${\tfrac{(a+h)f}{b} - f} = \tfrac{(a-b+h)f}{b}$.
Finally, $g > 0$ since $b < \tfrac{h}{2} < h$.
It is easy to check that the resulting parameters satisfy~\eqref{eq: three parts}.
Conversely, one can show that every choice of 
$a,b,c,e,f,g\in \mathbb{N}$ satisfying~\eqref{eq: three parts} occurs in this way.
\end{remark}

\begin{example}
In Theorem~\ref{thm: three parts}, 
taking $a=c=2$, $b=f=1$, and $e=g=3$ 
shows that $\mu = (7,7,4,3,3)$ is isoclinic.
Choosing $\delta = 1$ in Corollary~\ref{cor: isoclinic partition}
produces a totally symmetric $\operatorname{EITFF}_{\mathbb{R}}(d,r,n)$ with $d=10r$, $r=11\,660\,320\,672$, and $n=25$.
\end{example}

\begin{proof}[Proof of Theorem~\ref{thm: three parts}]
We apply Corollary~\ref{cor: isoclinic partition} with $\delta = 1$.
Enumerate $\lfloor \mu \rfloor = \{ (k_1,l_1),(k_2,l_2),(k_3,l_3) \}$ and $\mu^\uparrow = \{ \lambda_1,\lambda_2,\lambda_3,\lambda_4\}$ as in Figure~\ref{fig: label order}.
Taking $\delta = 1$ in~\eqref{eq: isoclinic partition}, the axial distances in the denominators are
\begin{align*}
D\big( \lambda_1 - \mu, (k_1,l_1) \big) &= a, & D\big( \lambda_3 - \mu, (k_1,l_1) \big) &= -(b+f+g), \\
D\big( \lambda_1 - \mu, (k_2,l_2) \big) &=  a+b+g, & D\big( \lambda_3 - \mu, (k_2,l_2) \big) &= -f, \\
D\big( \lambda_1 - \mu, (k_3,l_3) \big) &=  a+b+c+f+g, & D\big( \lambda_3 - \mu, (k_3,l_3) \big) &= c.
\end{align*}
By the hook length formula, each factor $\tfrac{ d_{\lambda} }{ n d_\mu }$ is the product of all hook lengths for $\mu$ divided by the product of all hook lengths for $\lambda$.
After applying the hook length formula and canceling common terms in the numerator and denominator, we find
\begin{align*}
\frac{d_{\lambda_1} }{ n d_\mu } &= \frac{ a(a+b+g)(a+b+c+f+g) }{ (a+g)(a+b+f+g)(a+b+c+e+f+g) }, \\[8 pt]
\frac{d_{\lambda_3} }{ n d_\mu } &= \frac{ c f (b+f+g) }{ (c+e)(b+f)(a+b+f+g) }.
\end{align*}
Notice the common factor of $a+b+f+g$ in the denominators.
With the elimination of this factor and~\eqref{eq: isoclinic partition} in mind, define
\[
y_p:= (a+b+f+g) \bigg( \frac{ d_{\lambda_1} }{n d_\mu} \cdot \frac{1}{ D\big( \lambda_1-\mu, (k_p,l_p) \big) } + \frac{ d_{\lambda_3} }{n d_\mu} \cdot \frac{1}{ D\big( \lambda_3-\mu, (k_p,l_p) \big) } \bigg),
\qquad
p \in [3].
\]
Explicitly, we compute
\begin{align*}
y_1 &=  \frac{ (a+b+g)(a+b+c+f+g) }{ (a+g)(a+b+c+e+f+g) } - \frac{ c f }{ (c+e)(b+f) } \\[8 pt]
y_2 &= \frac{ a(a+b+c+f+g) }{ (a+g)(a+b+c+e+f+g) } - \frac{ c (b+f+g) }{ (c+e)(b+f) } \\[ 8 pt]
y_3 &= \frac{ a(a+b+g) }{ (a+g)(a+b+c+e+f+g) } + \frac{ f (b+f+g) }{ (c+e)(b+f) }.
\end{align*}
By Corollary~\ref{cor: isoclinic partition}, $\mu$ is isoclinic if and only if $y_1 = y_3 = -y_2$.
After rearranging, we find $y_1 = y_3$ if and only if
\begin{equation}
\label{eq: three parts 1}
\frac{1}{(a+b+c+e+f+g)(a+g)} = \frac{f}{(a+b+g)(c+e)(b+f)},
\end{equation}
while $y_3 = -y_2$ if and only if 
\[
\frac{ (c-f)(b+f+g) }{ (c+e)(b+f) } = \frac{ a(2a+2b+c+f+2g) }{ (a+b+c+e+f+g)(a+g) }.
\]
In particular, if $\mu$ is isoclinic then 
$c - f > 0$.
Combining the above, we find $\mu$ is isoclinic if and only if 
$c-f > 0$, 
\eqref{eq: three parts 1} holds, and
\[
\frac{ (c-f)(b+f+g) }{ (c+e)(b+f) } = \frac{ af(2a+2b+c+f+2g) }{ (a+b+g)(c+e)(b+f) }.
\]
The latter equation may be rewritten as
\[
(c-f)\Big( a+b+g - \tfrac{a(c+f)}{c-f} \Big) \big( a+b+ f + g \big) = 0,
\]
and then (after canceling the positive factors and rearranging) as
\begin{equation}
\label{eq: three parts 3}
c = \frac{2af}{b+g} + f.
\end{equation}
Therefore, $\mu$ is isoclinic if and only if \eqref{eq: three parts 1} and \eqref{eq: three parts 3} hold.
After substituting \eqref{eq: three parts 3} and applying tedious simplification (perhaps with the aid of a computer algebra system), \eqref{eq: three parts 1} becomes
\[
e = \frac{(a+g)f}{b} - \frac{2af}{b+g} - f.
\]
Consequently, $\mu$ is isoclinic if and only if~\eqref{eq: three parts} holds.
\end{proof}

For isoclinic partitions with $m=4$ distinct parts, the algebra gets more difficult.
Things are simpler when the partition is symmetric, as below.

\begin{theorem}
\label{thm: four parts}
For $a,b,c,e \in \mathbb{N}$, the partition 
\[
\mu := \big({(a+b+c+e)}^a,{(a+b+c)}^b,{(a+b)}^c,a^e\big)\]
shown below (with some extra red boxes) is isoclinic if and only if 
\begin{equation}
\label{eq: four parts}
e = \frac{b^2}{c} + b.
\end{equation}

\begin{center}
\begin{tikzpicture}

\draw[thick] (0,0) -- (1,0) -- (1,1) -- (2,1) -- (2,2) -- (3,2) -- (3,3) -- (4,3) -- (4,4) -- (0,4) -- cycle;

\filldraw[thick,fill=ltblue] (1,0) rectangle (1-0.2,0+0.2);
\filldraw[thick,fill=ltblue] (2,1) rectangle (2-0.2,1+0.2);
\filldraw[thick,fill=ltblue] (3,2) rectangle (3-0.2,2+0.2);
\filldraw[thick,fill=ltblue] (4,3) rectangle (4-0.2,3+0.2);
\filldraw[thick,fill=ltpink] (1,1) rectangle (1+0.2,1-0.2);
\filldraw[thick,fill=ltpink] (3,3) rectangle (3+0.2,3-0.2);

\draw[<->] (-0.2,0+0.02) -- (-0.2,1-0.02);
\draw[<->] (-0.2,1+0.02) -- (-0.2,2-0.02);
\draw[<->] (-0.2,2+0.02) -- (-0.2,3-0.02);
\draw[<->] (-0.2,3+0.02) -- (-0.2,4-0.02);

\node[] at (-0.5,0.5) {$e$};
\node[] at (-0.5,1.5) {$c$};
\node[] at (-0.5,2.5) {$b$};
\node[] at (-0.5,3.5) {$a$};

\draw[<->] (0+0.02,4+0.2) -- (1-0.02,4+0.2);
\draw[<->] (1+0.02,4+0.2) -- (2-0.02,4+0.2);
\draw[<->] (2+0.02,4+0.2) -- (3-0.02,4+0.2);
\draw[<->] (3+0.02,4+0.2) -- (4-0.02,4+0.2);

\node[] at (0.5,4.5) {$a$};
\node[] at (1.5,4.5) {$b$};
\node[] at (2.5,4.5) {$c$};
\node[] at (3.5,4.5) {$e$};

\end{tikzpicture} 
\end{center}
\end{theorem}

\begin{remark}
Infinitely many positive integers 
$a,b,c,e$ 
satisfying~\eqref{eq: four parts} may be obtained as follows:
	\begin{itemize}
	\item[\textsc{Step 1:}]
	Choose any $b \in \mathbb{N}$.
	\item[\textsc{Step 2:}]
	Choose any positive divisor $c \mid b^2$.
	\item[\textsc{Step 3:}]
	Put $e:=\frac{b^2}{c} + b$.
	\item[\textsc{Step 4:}]
	Choose $a \in \mathbb{N}$ arbitrarily.
	\end{itemize}
Each choice yields an isoclinic partition, hence a totally symmetric and real EITFF.
\end{remark}

\begin{example}
In Theorem~\ref{thm: four parts}, taking $a=b=c=1$ and 
$e=2$
shows that $\mu = (5,3,2,1,1)$ is isoclinic.
Choosing $\delta = 0$ in Corollary~\ref{cor: isoclinic partition}
produces a totally symmetric $\operatorname{EITFF}_{\mathbb{R}}(42\,900,7700,13)$.
\end{example}

\begin{proof}[Proof of Theorem~\ref{thm: four parts}]
We apply Corollary~\ref{cor: isoclinic partition} with $\delta = 0$.
Enumerate 
\[
\lfloor \mu \rfloor = \{ (k_1,l_1),(k_2,l_2),(k_3,l_3),(k_4,l_4) \} 
\]
and $\mu^\uparrow = \{ \lambda_1,\lambda_2,\lambda_3,\lambda_4,\lambda_5\}$ as in Figure~\ref{fig: label order}.
Consider $L_0 = \{ \lambda_2,\lambda_4 \}$.
Since $\lambda_4 = \lambda_2'$, we have $d_{\lambda_4} = d_{\lambda_2}$, and the factor $\tfrac{ d_{\lambda} }{ n d_\mu }$ in~\eqref{eq: isoclinic partition} does not depend on $\lambda$.
Thus, $\mu$ is isoclinic if and only if
\[
\exists \, \gamma \geq 0 \text{ s.t.\ }
\frac{1}{D\big( \lambda_2 - \mu, (k_q, l_q) \big) }
+
\frac{1}{D\big( \lambda_4 - \mu, (k_q, l_q) \big) }
= (-1)^{q} \gamma
\qquad
\forall q \in [m].
\]
The axial distances in the denominators are
\begin{align*}
D\big( \lambda_2 - \mu, (k_1,l_1) \big) &= -e, & D\big( \lambda_4 - \mu, (k_1,l_1) \big) &= -(2b+2c+e), \\
D\big( \lambda_2 - \mu, (k_2,l_2) \big) &=  b, & D\big( \lambda_4 - \mu, (k_2,l_2) \big) &= -(b+2c), \\
D\big( \lambda_2 - \mu, (k_3,l_3) \big) &=  b+2c, & D\big( \lambda_4 - \mu, (k_3,l_3) \big) &= -b, \\
D\big( \lambda_2 - \mu, (k_4,l_4) \big) &=  2b+2c+e, & D\big( \lambda_4 - \mu, (k_4,l_4) \big) &= e.
\end{align*}
Substituting above, we find $\mu$ is isoclinic if and only if
\[
\frac{1}{e} + \frac{1}{2b+2c+e} = \frac{1}{b} - \frac{1}{b+2c}.
\]
After rearranging and solving for 
$e$, 
we conclude $\mu$ is isoclinic if and only if 
$e = \frac{b^2}{c} + b$ 
or 
$e = -\frac{b^2+2c^2+2bc}{c}$.
The latter is not possible since 
$b,c,e>0$.
\end{proof}

\begin{example}
Not every isoclinic partition with $m=4$ distinct parts takes the symmetric form in Theorem~\ref{thm: four parts}.
For example, $\mu = (12^2,7^3,5,4^3)$ is not symmetric, but one can check that it satisfies~\eqref{eq: isoclinic partition} and is therefore isoclinic.
More generally, for every integer $a \geq 1$, one can show the partition $\big( (12a)^2, (7a)^3, 5a, (4a)^3 \big)$ is isoclinic.
\end{example}

A symmetric isoclinic partition with $m=6$ distinct parts appears in~\cite{FIJM:ICASSP}.
The proof of its existence is complicated and involves elliptic curves.
We leave the following more general problem for future research.

\begin{problem}
Are there isoclinic partitions with $m$ distinct parts for every $m\geq 1$?
\end{problem}

\section*{Acknowledgments}
JWI was supported by a grant from the Simons Foundation, by NSF DMS 2220301, and by the Air Force Institute of Technology through the Air Force Office of Scientific Research Summer Faculty Fellowship Program, Contract Numbers FA8750-15-3-6003, FA9550-15-0001 and FA9550-20-F-0005.
JJ was supported by NSF DMS 2220320.
DGM was supported by NSF DMS 2220304.
The views expressed are those of the authors and do not reflect the official guidance or position of the United States Government, the Department of Defense,  the United States Air Force, or the United States Space Force.

\bigskip

\bibliographystyle{abbrv}
\bibliography{refs}

\end{document}